\documentclass{amsart}

\usepackage{amssymb,amscd}

\newtheorem{thm}{Theorem}[section]
\newtheorem{lem}[thm]{Lemma}
\newtheorem{prop}[thm]{Proposition}
\newtheorem{cor}[thm]{Corollary}

\newtheorem{rem}[thm]{Remark}

\input amssym.def

\def\wt{\widetilde}

\def\ff{\frak}
\def\Spec{\mbox{\rm Spec}}

\def\Hom{\mbox{ Hom}}
\def\rank{\mbox{ rank}}
\def\Ext{\mbox{ Ext}}
\def\Ker{\mbox{ Ker }}

\def\l{\langle\:}
\def\r{\:\rangle}

\def\cal{\mathcal}

\def\wt{\widetilde}

\def\ff{\frak}
\def\Spec{\mbox{\rm Spec}}

\def\Hom{{\rm Hom}}

\def\rank{\mbox{\rm rank}}
\def\Ext{{\rm Ext}}
\def\embdim{{\mbox{\rm emb.dim } }}
\def\trdeg{{\mbox{\rm tr.deg} }}

\def\Ker{\mbox{\rm Ker }}

\def\l{\langle\:}
\def\r{\:\rangle}

\def\cal{\mathcal}

\begin{document}

\title[One-dimensional bad Noetherian domains]{One-dimensional bad Noetherian domains}
%\thanksref{T1}}  %Intersections of valuation overrings of a two-dimensional
%Noetherian domain}
%\thanks{\today}

%\thanks[T1]{This research was partially supported by grant ? from the
%National Security Agency.}
%\keywords{Stable domain, generalized Dedekind domain, Pr\"ufer domain}
\author{Bruce Olberding}

%%%%%%%%%%%%%%%%%%%%%%%%%%%%%%%%% \ead{olberdin@nmsu.edu}

\address{Department of Mathematical Sciences, New Mexico State University,
Las Cruces, NM 88003-8001}
%\begin{document}

\maketitle

%\begin{abstract}  Let $D$ be a Noetherian domain of Krull dimension
%2.  A 1969 theorem of W. Heinzer states that every finite character
%intersection of rank one discrete valuation overrings of $D$ (i.e.
%every Krull overring of $D$)  is a Noetherian domain.  We classify
%the integrally closed overrings of $D$ that arise as finite
%character intersections of arbitrary valuation overrings of $D$.
%\end{abstract}

\begin{abstract} Local Noetherian domains arising as local rings of points of varieties or in the context of algebraic number theory are analytically unramified, meaning their completion has no nontrivial nilpotent elements.  However, looking elsewhere, many sources of  analytically ramified local Noetherian have been exhibited over the last seventy five years.  
%In  dimension $1$, a local ring is analytically ramified if and only if it has finite normalization.  In dimension $1$ then, basic ideas reduce the study of analytically ramified local domains to those that have normalization a discrete rank one valuation ring (DVR).  
We give a unified approach to  a number of such examples by  
describing classes of DVRs which occur as the normalization of an  analytically ramified   
local Noetherian domain, as well as some that do not occur as such a normalization.  We parameterize these examples, or at least large classes of them, using the module of K\"ahler differentials of a relevant field extension.          
\end{abstract}

%%%%%%%%%%%%%%%%%%%%%%%%%%%%%%%%%%%%\end{frontmatter}

%%%%%%%%%%%%%%%%%%%%%%%%%%%%%%%%%%%%%%%%%%%%%%%%%%%%%%%%%%%%%%%%%
%%%%%%%%%%%%%%%%%%%%%%%%%%%%%%%%%%%%%%%%%%%%%%%%%%%%%%%%%%%%%%%%%
\section{Introduction}

As early as 1935, examples were known of Noetherian rings  differing in  fundamental ways  from the Noetherian rings arising in algebraic geometry or algebraic number theory. %the rings which motivate and drive  the subject of commutative algebra.  
In that year, fourteen years after Noether introduced the ascending chain condition into ideal theory, Akizuki and Schmidt each gave examples of one-dimensional local Noetherian domains failing to have {finite normalization}, where a ring $A$ has {\it finite normalization} if its integral closure in its total ring of  quotients is a finitely generated $A$-module \cite{Akizuki, Schmidt}.  
 Hence  by a theorem of Krull, since the rings of Akizuki and Schmidt are one-dimensional and do not have finite normalization, they  have   nontrivial nilpotents in the completion; that is, the rings are {\it analytically ramified} \cite{Krull}.    Over the last seventy five years,   other constructions of one-dimensional analytically ramified local Noetherian domains have been developed; see for example \cite{Akizuki, Bennett, deL, FR, GL,  HRS, Leq, Na, OlbGFF, Reg, Schmidt}\footnote{Although of a different spirit, Lech's construction in \cite{Lech} of local Noetherian domains having prescribed completion can also be used to produce analytically ramified local Noetherian domains.  Where the references cited above differ is in an emphasis  on prescribed normalization rather than completion. For a  survey that touches on each of the cited references, see \cite{OlbFinite}.}.     
 The most well-known example perhaps  is   in Example 3 of  an appendix of Nagata's {\it Local Rings} titled, ``Examples of bad Noetherian rings.''  In this article we seek to unify a number of these examples by considering the question of which rank one discrete valuation rings (DVRs) can occur as the normalization of an analytically ramified local Noetherian domain.

In considering these examples, certain themes predominate.  For example, the constructions of Akizuki, Schmidt and Nagata all occur in an immediate extension  of DVRs.  More precisely, the examples of Akizuki, Schmidt and Nagata are of an analytically ramified local Noetherian ring $R$ such that  $U \subseteq R \subseteq  V$, where $U$ and $V$ are  DVRs with $V = U + {\ff M}_UV$ (here, as throughout the paper,  ${\ff M}_U$ denotes the maximal ideal of the valuation ring  $U$), and $V$ is the integral closure of $R$ in its quotient field (i.e., $V$ is the {\it normalization} of $R$).  The choice of immediate extension $U \subseteq V$ is different for each of these authors; for example, Schmidt and Nagata require characteristic $p$, whereas Akizuki does not.  On the other hand, the examples of Schmidt and Nagata are simple integral extensions of the DVR $U$, while Akizuki's example is not even finite over $U$.      

One of the main goals of  this article is to  formalize what it is about the choices of $U$ and $V$ that permits analytically ramified local Noetherian rings to be situated between them.  
    This proves to depend only on the quotient fields $Q$ and $F$ of $U$ and $V$, respectively:  In Corollary~\ref{main theorem 1}, we show there exists an analytically ramified local Noetherian domain containing $U$ and having normalization $V$ if and only if either (a) $F$ has characteristic $0$ and $F/Q$ is not algebraic, or (b) $F$ has characteristic $p>0$ and $F \ne Q[F^p]$.  Statements (a) and (b) can be compressed into the single assertion that the module of K\"ahler differentials $\Omega_{F/Q}$  for the field extension $F/Q$ does not vanish.  In fact, we show in Theorem~\ref{DVR correspondence} that each proper full $V$-submodule $K$ of $\Omega_{F/Q}$ yields a different analytically ramified local Noetherian domain $R$ with normalization $V$, the ring  $R$ consisting of ``antiderivatives'' for $K$ in $V$; that is, $R = V \cap d^{-1}_{F/Q}(K)$, where 
    $d_{F/Q}:F \rightarrow \Omega_{F/Q}$ is the exterior differential of the field extension $F/Q$.  
    
    While not every one-dimensional analytically ramified 
    local Noetherian domain containing $U$ and having  normalization $V$ is captured in this way, it is possible using Theorem~\ref{DVR correspondence} to describe precisely the rings obtained from the exterior differential: They are the  analytically ramified local Noetherian domains $R$ such that every ideal of $R$ has a principal reduction of reduction number at most $1$; equivalently, $\widehat{R}$ has a nonzero prime ideal $P$ such that $P^2 = 0$ and $\widehat{R}/P$ is a DVR.   Following Lipman \cite{Lipman} and Sally and Vasconcelos \cite{SV}, such rings are termed {\it stable}.  Stable rings are discussed in Section 2.  
    
    Although a one-dimensional analytically ramified local Noetherian domain $R$ need not be stable,  such a ring must have an analytically ramified stable ring between it and its quotient field (see the discussion that precedes Corollary~\ref{main theorem 1}).   This observation, along with the material in Sections 4 and 5, is applied in Sections 6 and 7 to  the problem of determining which DVRs occur as the normalization of an analytically ramified local Noetherian domain.   For example, we show that any uncountable DVR of equicharacteristic $0$ (e.g., a complete DVR of equicharacteristic $0$)  is for each $d>0$ the normalization of an analytically ramified local Noetherian stable domain of embedding dimension $d$ (Corollary~\ref{dig cor}).
Similar statements hold when $V = \widehat{{\mathbb{Z}}}_{(p)}$ or   $V = k[[X]]$ for a field $k$ of positive characteristic $p$ such that $[k:k^p]$ is uncountable (Theorem~\ref{catalog complete DVR}(2)).  By contrast, if $k$ is a perfect field of positive characteristic, then $k[[X]]$ is not the normalization of an analytically ramified local Noetherian domain (Theorem~\ref{catalog complete DVR}(1)). In general, a DVR $V$ of positive characteristic $p$ with quotient field $F$ is the normalization of an analytically ramified local Noetherian domain if and only if there is a subfield $Q$ of $F$ such that $F = Q + V$ and $F \ne Q[F^p]$ (Theorem~\ref{Bennett case}).  All of these results concern raw existence of an analytically ramified local Noetherian domain with prescribed normalization, but in the case of an algebraic function field $F/k$  we are able to assert existence within the category of $k$-algebras: In functions fields the valuation rings with finitely generated residue field that are the normalization of an analytically ramified local Noetherian domain in the function field are precisely the non-divisorial DVRs (Theorem~\ref{function field}).

  Throughout the article we focus on the unibranched case: local Noetherian domains having normalization  a DVR.    
       In general, basic ideas allow one to reduce things to consideration of one-dimensional analytically ramified local Noetherian domains $R$ having normalization  a DVR; see for example \cite[Section 1]{Bennett} or the proof of Theorem 5.10 in \cite{OlbGFF}.  

\medskip

All rings considered are commutative and have an identity. In addition, we use the following standard notation and terminology.  The {\it normalization} $\overline{R}$ of a ring $R$ is the integral closure of $R$ in its total ring of quotients.  The ring $R$ has {\it finite normalizaton} if $\overline{R}$ is a finitely generated $R$-module. When $R$ is a quasilocal ring with maximal ideal $M$, then we denote by $\widehat{R}$ the completion of $R$ in the $M$-adic topology.  
 When $V$ is a valuation domain, we denote its maximal ideal by ${\ff M}_V$.  An extension of valuation domains $U \subseteq V$ is  {\it immediate} if $U$ and $V$ share the same value group and residue field.  If also $U$ and $V$ are DVRs, then $U \subseteq V$ is immediate if and only if  ${\ff M}_U \subseteq {\ff M}_V$ and $V = U + {\ff M}_UV$; if and only if $V/U$ is a torsion-free divisible $U$-module.    
 
\section{Preliminaries on derivations and stable rings}

In this section we review some technical properties of derivations that are needed later, and we discuss briefly the class of one-dimensional stable rings.  
Let $S$ be a ring, $R$ a subring of $S$ and $L$  an $S$-module.  Given a subset $A$ of $S$, an  {\it $A$-linear derivation} $D:S \rightarrow L$ is an $A$-linear mapping such that $D(st) = sD(t) + tD(s)$ for all $s,t \in S$.  The main properties of derivations we need are collected in (2.1) and (2.2).

\smallskip

{\bf (2.1)}  {\it The module $\Omega_{S/A}$ of K\"ahler differentials}.  Let $A \subseteq S$ be an extension of rings.   There exists an $S$-module
 $\Omega_{S/A}$,  along with an $A$-linear derivation $d_{S/A}:S \rightarrow \Omega_{S/A}$, such that for every
  derivation $D:S \rightarrow L$, there is a unique $S$-module
  homomorphism $\alpha:\Omega_{S/A} \rightarrow L$ with  $D =
  \alpha \circ d_{S/A}$; see for example, \cite[Theorem 1.19, p.~12]{Kunz}.  
  The actual construction of $\Omega_{S/A}$ is not needed here, but we do use the fact that 
   the image of $d_{S/A}$ in
  $\Omega_{S/A}$ generates $\Omega_{S/A}$ as an $S$-module \cite[Remark 1.21, p.~13]{Kunz}.
The $S$-module $\Omega_{S/A}$  is the {\it module of K\"ahler differentials}
 of the
 ring extension $A \subseteq S$, and the derivation  $d_{S/A}:S \rightarrow \Omega_{S/A}$ is the {\it  exterior differential} \index{exterior differential} of $A \subseteq S$.
 
\smallskip

{\bf (2.2)}
{\it K\"ahler differentials of field extensions}.  
When $F/Q$ is an extension of fields, then $\Omega_{F/Q}$ is an $F$-vector space, and hence its structure is determined entirely by its dimension, $\dim_F \Omega_{F/Q}$.  If $F/Q$ is a separably generated extension, then $\dim_{F} \Omega_{F/Q} = {\rm{tr.deg}}_Q F$ \cite[Corollary 5.3, p.~74]{Kunz}. If $Q$ has characteristic $p \ne 0$, 
then $\dim_{F}\Omega_{F/Q}$ is the cardinality of   a $p$-basis   of $F$ over $Q$ \cite[Proposition 5.6, p.~76]{Kunz}.  
Of particular interest  is the case  where $\Omega_{F/Q} = 0$.   
 Recall that for a field $F$ of
characteristic $p \ne 0$, $F^p$ denotes the subfield consisting of
the $p$-powers of elements in $F$, and if $Q$ is a subfield of $F$,
then $Q[F^p]$ is the subfield of $F$ generated by $Q$ and $F^p$.
Using the above ideas, along with the existence of a $p$-basis, the case where $\Omega_{F/Q}=0$ can be discerned:
When $F$ has characteristic $0$,  $\Omega_{F/Q} =
0$ if and only if $F$ is  algebraic over $Q$; otherwise, when $F$ has prime characteristic $p$,
$\Omega_{F/Q} = 0$ if and only if $F = Q[F^p]$
\cite[Proposition 5.7]{Kunz}.

\smallskip

 Following Lipman \cite{Lipman} and Sally and Vasconcelos \cite{SV}, an ideal $I$ of a ring $R$ is {\it stable} if $I$ is projective over its ring of endomorphisms.  If every ideal of $R$ containing a nonzerodivisor is stable, then $R$ is a {\it stable ring}.    When $R$ is a quasilocal domain (the main case which we consider in the article), stable ideals $I$ are characterized by the property that there exists $i \in I$ such that $I^2 = iI$ \cite[Lemma 3.1]{OlStructure}.  Thus a quasilocal domain $R$ is stable if and only if every ideal has a principal reduction of reduction number $\leq 1$.  
 \smallskip
 
{\bf (2.3)}  {\it Noetherian stable rings}.  Since every ideal of a quasilocal stable domain has a principal reduction, and hence the radical of every proper ideal is the radical of a principal ideal,  it follows that  a
  local Noetherian stable ring has Krull dimension $1$.   A local Noetherian  ring $R$ with finite normalization    is a stable ring if and only $R$ is a {\it $2$-generator ring}, meaning that every ideal of $R$ can be generated by $2$ elements \cite[Theorem 2.4]{SV}.  However, if $R$ does not have finite normalization, then $R$ need not be a $2$-generator ring.  Examples of such rings were given by Sally and Vasconcelos \cite{SV} and Heinzer, Lantz and Shah \cite{HLS} using a construction of Ferrand and Raynaud over a specific field of characteristic $2$.  We will exhibit large classes of stable rings without finite normalization in this article; other sources of examples can be found in \cite{OlbGFF}. 
  \smallskip    
  
 {\bf (2.4)} {\it Bad stable rings.}  Among stable rings, our main focus is on the class of one-dimensional quasilocal stable domains $R$ without finite normalization $\overline{R}$.  We refer to such rings as {\it bad stable rings}, where ``bad'' is in the sense of Nagata's appendix, ``Examples of bad Noetherian rings,'' in \cite{Na}.  There is a helpful extrinsic characterization of bad stable rings in terms of the extension $R \subseteq \overline{R}$: A quasilocal domain $R$ is a bad stable ring if and only if 
   $\overline{R}$  is a DVR, $\overline{R}/R$ is a divisible $R$-module and every $R$-submodule of $\overline{R}$ containing $R$ is a ring; see \cite[Proposition 2.1]{OlbGFF}. In the next section, we express the last condition by saying  
    $R \subseteq \overline{R}$ is a {\it quadratic extension.}  Using this characterization, it is shown in \cite[Theorem 4.2]{OlbGFF} that a quasilocal domain $R$ with quotient field $F$ is a bad stable ring if and only if  
$\overline{R}$ is a DVR and $\overline{R}/R \cong \bigoplus_{i \in I} F/\overline{R}$ for some index set $I$.  This in turn yields a characterization of bad Noetherian stable rings of embedding dimension $n$ as those for which the index set $I$ can be chosen to have $n-1$ elements \cite[Corollary 4.3]{OlbGFF}.  In particular, $R$ is a {\it bad $2$-generator ring} (that is, a local $2$-generator ring without finite normalization) if and only if $\overline{R}$ is a DVR and $\overline{R}/R \cong F/\overline{R}$.

\smallskip

{\bf (2.5})  {\it The completion of a bad stable ring.}  Corollary~\ref{stable corollary} shows that a bad stable ring need not be Noetherian.     This subtlety forces us to frame the 
 analytic characterization of bad stable rings in terms of the {\it completion of $R$ in the  ideal topology},  $\wt{R} = \varprojlim R/rR$, where $r$ ranges over all nonzero elements of $R$.  When $R$ is one-dimensional Noetherian, or stable, it is easy to see that $\wt{R}$ is isomorphic to the usual completion $\widehat{R}$ of $R$ in its $M$-adic topology.  In any case, a quasilocal domain $R$ is a bad stable ring if and only if  
there is a nonzero prime ideal $P$ of $\wt{R}$ with $P^2 =0$ and  $\wt{R}/P$ is a DVR \cite[Theorem 3.4]{OlbGFF}. Thus, among Noetherian rings, the  
   bad stable domains are characterized by the property that 
  there exists a nonzero prime ideal $P$ of the $M$-adic completion $\widehat{R}$ of $R$ such that $P^2 = 0$ and $\widehat{R}/P$ is a DVR.

%\begin{cor}  \label{stable ring cor} If $R$ is a bad stable ring, then there exists a nonzero prime ideal $P$ of $\widehat{R}$ such that $P^2 = 0$ and $\widehat{R}/P$ is a DVR.
%\end{cor} 

\section{Analytic isomorphisms}

The purpose of this section is prove some general results that in the next section will specialize to some of our main technical tools for dealing with rings in an immediate extensions of DVRs.  
  The justifications for the results in this section involve some   technical but elementary arguments which we give in a very general setting.  Our     
 motivation for working at this  level of abstraction is twofold.  First, it comes at no extra expense of proof, since our particular setting involving DVRs does not seem to simplify or shorten much the arguments.  Second, in \cite{OlbFR}, we apply the lemmas as stated here to a more general context in which this level of abstraction is needed.

Let $\alpha:K \rightarrow L$ be a homomorphism of $A$-modules, and let $C$ be a multiplicatively closed subset of $A$ consisting of nonzerodivisors of $A$.  Following Weibel in \cite{Wei}, we say that $\alpha$ is an {\it analytic isomorphism along $C$} if for each $c \in C$, the induced mapping $\alpha_c:K/cK \rightarrow L/cL:a +cK\mapsto \alpha(a) + cL$ is an isomorphism.  It follows that $\alpha$ is analytic along $C$ if and only if $\alpha(K) \cap cL = c\alpha(K)$ and $L = \alpha(K) + cL$ for all $c \in C$.  We express the former condition by saying that $L/\alpha(K)$ is a {\it $C$-torsion-free} $A$-module, and the latter by saying that $L/\alpha(K)$ is a {\it $C$-divisible} $A$-module.  Similarly, an $A$-module $T$ is {\it $C$-torsion} provided that for each $t \in T$, there exists $c \in C$ with $ct =0$.

%To avoid repeating hypotheses, we make the following standing assumptions for this section only.
%\medskip

%\begin{quote} {\it
%{\noindent}$A \subseteq S$
%is an extension of rings; $C$ is a multiplicatively closed subset of $A$ consisting of nonzerodivisors of $S$;  the inclusion mapping $A \rightarrow S$ is an analytic isomorphism along $C$; and $d$ is the exterior differential $d_{S_C/A_C}$.   } 
%\end{quote}

%\medskip

An extension $R \subseteq S$ is {\it quadratic} if every $R$-submodule of $S$ containing $R$ is a ring; equivalently, $st \in sR + tR + R$ for all $s,t \in S$. In light of  the fact that the inclusion of a bad stable ring in its normalization is a quadratic extension (2.4), the following lemma, taken from \cite[Lemma 3.2]{OlbGFF}, is relevant in a number of results that follow.

\begin{lem} \label{start} Let $R \subseteq S$ be an extension of rings, and let $C$ be a multiplicatively closed subset of $R$ consisting of nonzerodivisors of $S$.   If the $R$-module $S/R$ is $C$-torsion and $C$-divisible,
 then  the following statements are equivalent.

\begin{itemize}

\item[(1)] $R \subseteq S$ is a quadratic extension of rings. \index{quadratic extension}

%\item[(2)]  For all $s \in S$, $(R:_R s) = (R:_R s)S \cap R$.

%Every $R$-submodule  of $S$ containing $R$ is a ring.

\item[(2)]  There exists an
$S$-module $T$ and a  derivation $D:S \rightarrow T$ with  $R =
\Ker D$.

%\item[(5)] The mapping $S/R \rightarrow \Omega_{S/R}$ induced by $d_{S/R}$ is an   isomorphism of $R$-modules.

% \item[(6)]  For each $c \in C$ and ideal $I$ of $R$ containing $c$, $(I \cap cS)^2 \subseteq cI$.

\item[(3)]  For all $c \in {{{{\it C}}}}$, $(R \cap cS)^2 \subseteq cR$.

\item[(4)]  $S/R$ admits an $S$-module structure extending the  $R$-module structure.     

%\item[(8)]  For all ideals $I$ of $R$ meeting ${{{{\it C}}}}$, $(R \cap
%IS)^2 \subseteq I$.

%\item[(9)]  For each $c \in {{{{\it C}}}}$, the $R$-module $(cS \cap
%R)/cR$ admits an $S$-module structure.

\end{itemize}

\end{lem}

The next lemma, which  establishes the existence of a derivation
associated to certain subrings between $A$ and $S$,  is the key technical
result needed to frame the correspondence in Proposition~\ref{correspondence} and for immediate extensions of DVRs in Theorem~\ref{DVR correspondence}.  The proof uses the notion of {\it Nagata idealization} (or {\it trivialization}) of a module.  Let $B$ be a ring and $L$ be a $B$-module.  Then the idealization $B \star L$ of $L$ is defined as a $B$-module to be $B \oplus L$, and whose ring multiplication is given by $(b_1,\ell_1)\cdot(b_2,\ell_2) = (b_1b_2,b_1\ell_2 +b_2\ell_1)$ for all $b_1,b_2 \in B$ and $\ell_1,\ell_2 \in L$.  Thus the ring $B \star L$ has a square zero ideal corresponding to $L$.  
  In analogy with internal direct sums, we can consider decompositions of rings into idealizations.  Thus if $A$ is a subring of $B$ and $I$ is an ideal of $B$, then we write $B = A \star I$ when $B = A + I$, $0 = A \cap I$, and $I^2 = 0$.   It is straightforward to see that the existence of such a decomposition is equivalent to the presence of a  derivation:
{\it  $B = A \star I$ if and only if there exists a
derivation $D:B \rightarrow I$ such that $A = \Ker D$ and $D(i) = i$
for all $i \in I$}; cf.~\cite[Proposition 1.17, p.~11]{Kunz}.

\begin{lem} \label{pre splitting lemma} \label{split completion}
 Let $A \subseteq R \subseteq  S$
be an extension of rings such that $R \subseteq S$ is a quadratic extension, let $C$ be a multiplicatively closed subset of $A$ consisting of nonzerodivisors of $S$, and suppose that   the inclusion mapping $A \rightarrow S$ is an analytic isomorphism along $C$. 
  If $S/R$ is $C$-torsion, then there exists an $A$-linear derivation $D:S \rightarrow L$,
where $L$ is a ${{{{\it C}}}}$-torsion-free $S$-module, such that
$D(S) = cD(S)$ for each $c \in {{{{\it C}}}}$, and  $R = D^{-1}(K)$
for some $S$-submodule $K$ of $L$ such that $L/K$ is $C$-torsion.
\end{lem}

\begin{proof}
%(1) $\Rightarrow$ (2)  For each $c \in {{{{\it C}}}}$ we have by (1) that
%$R/cR = A_c/cR \oplus (cS \cap R)/cR$.  Hence $R = A_c + cS \cap R$
%and $A_c \cap cS \subseteq cR$, and it remains to show that $S = A_c
%+ cS$. In fact, since $R = A_c + cS \cap R$, it is enough to show
%that $S \subseteq R + cS$. To this end, let $c \in {{{{\it C}}}}$ and $s
%\in S$. Then there exists $d \in {{{{\it C}}}}$ such that $ds \in R$.
%Since $R = A_{cd} + cdS \cap R$, there exists $a_{cd} \in A_{cd}$
%and $\sigma_{cd} \in S$ such that $ds = a_{cd} + cd\sigma_{cd}$. Now
%$A_{cd} \subseteq A_d$, so $d(s - c\sigma_{cd}) = a_{cd} \in A_{cd}
%\cap dS \subseteq A_d \cap dS \subseteq dR$. Hence $s - c\sigma_{cd}
%\in R$, so $s \in R + cS$, as claimed.  Thus $S \subseteq A_c + cS$. 
Define $\wt{R} = \varprojlim R/cR$ and $\wt{S} = \varprojlim S/cS$, where $c$ ranges over the elements of $C$.  We denote the elements of $\wt{R}$ by $\l r_c + cR \r$, where each $r_c \in R$ and for $c,d \in C$, $r_c - r_{cd} \in cR$.  The elements of $\wt{S}$ are represented similarly as  $\l s_c + cS\r$. 
Let $I$ be the kernel of the canonical ring homomorphism
$\phi:\wt{R} \rightarrow \wt{S}$; i.e., $$I = \{ \l r_c + cR \r \in \wt{R}:r_c \in cS {\mbox{ for all }} c \in C\}.$$ Define  a subring $A'$ of $\wt{R}$ by
 $$A' = \{ x
\in \wt{R}:x = \l a_c + cR \r,
 a_c \in A {\mbox{ for all }} c \in {{{{\it C}}}}\}.$$
%
%$$A' = \{ \l a_c + cR \r
%\in \wt{R}:a_c \in A {\mbox{ for all }} c \in {{{{\it C}}}}\}.$$
  We first prove the following
claim.
\begin{quote}{\it There is a derivation $D_0: \wt{R} \rightarrow I$ such
that $A'=\Ker D_0$  and $D_0(x) = x$ for all $x \in I$.} \end{quote}  By the observation preceding the lemma,  it is enough  to show that
$\wt{R} = A' \star I$.  Let $\l r_c + cR \r \in \wt{R}$.  Then since
$S/A$ is $C$-divisible, for each $c \in C$ there exist $a_c \in A$
and $s_c \in S$ such that $r_c = a_c + cs_c$.  We claim that $\l a_c
+ cR \r \in A'$ and $\l cs_c + cR \r \in I$.  Since each $a_c \in
A$, the first assertion amounts to showing that $\l a_c + cR \r \in
\wt{R}$. Let $c,d \in C$.  Then since $r_c -r_{cd} \in cR$, we use
the fact that $S/A$ is $C$-torsion-free to obtain $$a_c -
a_{cd} = r_c - r_{cd} + cds_{cd} - cs_c \in cS \cap A \subseteq
cR,$$ and hence $\l a_c + cR \r \in \wt{R}$.  This also shows that
$cds_{cd} - cs_c \in cR$, and hence that $\l cs_c + cR \r \in
\wt{R}$. Therefore, $\l cs_c + cR \r \in I$, and we have proved that
$\wt{R} = A' + I$. Using again the fact that $S/A$ is
$C$-torsion-free, we  conclude that $\wt{R} = A' \oplus I$.
For if $\l a_c + cR \r \in A' \cap I$, then for each $c \in C$, $a_c \in cS \cap A \subseteq cR$, so that  $\l a_c +
cR\r =0$.

Thus to see that $\wt{R} = A' \star I$, it remains to show that $I^2 = 0$.
Let $\l cs_c + cR \r, \l cs'_c + cR \r \in I$, where  $s_c,s'_c
\in S$ with $cs_c,cs'_c \in R$. By assumption, $R\subseteq S$ is quadratic and the $R$-module $S/R$ is $C$-torsion.  Also, since $S/A$ is $C$-divisible, so is $S/R$.  Therefore,  Lemma~\ref{start} is applicable, so that by (3) of this lemma, for each $c \in C$ we have $(cs_c)(cs'_c)
\in (cS \cap R)^2 \subseteq cR$, and hence $\l cs_c + cR \r \l cs'_c
+ cR \r = 0$. Therefore, $I^2 = 0$, and we conclude that $\wt{R} =
A' \star I$. This proves  the claim that there
is a derivation $D_0: \wt{R} \rightarrow I$ such that $A' = \Ker
D_0$ and $D_0(x) = x$ for all $x \in I$.

We now prove the lemma.  With $D_0$ and $I$ as above, define an $A$-linear
derivation $D:R \rightarrow I$ by $D(r) = D_0(\l r + cR \r)$ for all
$r \in R$. 
 Since $\wt{R}$ is ${C}$-torsion-free
 and $I \subseteq \wt{R}$, then
$I$ is also ${{{{\it C}}}}$-torsion-free. We extend
$D$ to a $C$-linear derivation $D:R_{{{{\it C}}}} \rightarrow I_{{{{\it C}}}}$
 by defining for all $r \in R$ and
$c \in {{{{\it C}}}}$:
$$D\left(\frac{r}{c}\right) = \frac{1}{c}D(r).$$
We view $I$ as an $S$-module in the following way. For each $s \in S$ and $\l r_c +cR \r \in I$,  define $s \cdot \l r_c +cR \r = \l a_cr_c + cR \r$, where for each $c \in C$, $s = a_c + c\sigma_c$, with $a_c \in A$ and $\sigma_c \in S$. (We are using here the fact that $S/A$ is $C$-divisible.) The operation $\cdot$ is well-defined, since if for some $c \in C$, $s = b + c\tau$, with $b \in A$ and $\tau \in S$, then by Lemma~\ref{start}, $a_cr_c - br_c = (\tau - \sigma_c)cr_c \in (cS \cap R)^2 \subseteq cR$. (Here we are using that $\l r_c + cR \r \in I$ implies $r_c \in cS \cap R$.) It is now easily checked that $\cdot$ defines an $S$-module structure on $I$ that extends the $R$-module structure on $I$.  

We define $K$ to be the $S$-submodule of $I$ generated by $D(R)$.  
Since $S \subseteq R_C$, we may apply $D$ to $S$,  and define $L$ to be the $S$-submodule of $K_{{{{\it C}}}}$ generated by
$D(S)$.  Observe that $L/K$ is $C$-torsion since $L \subseteq
K_C$. Moreover,
$D|_{S}:S \rightarrow L$ is a derivation that extends the  derivation $D|_{R}:R
\rightarrow K$.

To see that $R = S \cap D^{-1}(K)$, suppose that $s \in S$ such that $D(s)
\in K$. Then since $S/R$ is $C$-torsion, there exists $d \in {{{{\it
C}}}}$ such that $ds \in R$, and, since $D$ is ${{{{\it
C}}}}$-linear, $D(ds) = dD(s) \in dK$. Thus since $K \subseteq
\wt{R}$, there exists $\l r_c + cR \r \in \wt{R}$ with $D(ds) = \l
dr_c + cR \r$.  Now as we have shown, $\wt{R} = \Ker D_0 + I$, so
$$\l ds + cR \r = \l a_c + cR \r + \l c\sigma_c + cR\r,$$ where $\l
a_c + cR \r \in \Ker D_0$ and $\l c\sigma_c + cR \r \in I$ (and each
$\sigma_c$ is in $S$ with $c\sigma_c \in R$). Thus since $D_0(x) = x$ for all $x \in I$,  $$D(ds) = D_0(\l ds
+cR \r) = \l c\sigma_c + cR \r,$$ and pairing this with our
previous representation of $D(ds)$, we conclude that $\l c\sigma_c +
cR \r = \l dr_c + cR\r$. In particular, $d\sigma_d - dr_d \in dR$,
and since $d$ is a nonzerodivisor in $S$, $\sigma_d - r_d \in R$.
Consequently, $\sigma_d \in R$.   Moreover, from $\l ds + cR \r = \l a_c
+ cR \r + \l c\sigma_c + cR \r$, we see that  there exists $\rho \in
R$ such that $ds - d\sigma_d - d\rho \in A \cap dS \subseteq  dR$
(here we are using that $S/A$ is $C$-torsion-free). Since $d$ is a
nonzerodivisor in $S$, this implies that $s - \sigma_d - \rho \in
R$. Since $\rho \in R$ and we have shown that $\sigma_d \in R$, we have then that $s
\in R$. Therefore, $R = S \cap D^{-1}(K)$.

Finally, to complete the proof, we show that $D(S)$ is ${{{{\it
C}}}}$-divisible.  Let $s \in S$ and $c \in {{{{\it C}}}}$. Since
$S/A$ is $C$-divisible, there exists $a \in A$ such that $s = a +
c\sigma$ for some $\sigma \in S$.  Since  $$\Ker D_0 = A' =
 \{ \l a_c + cR \r
\in \wt{R}:
 a_c \in A {\mbox{ for all }} c \in {{{{\it C}}}}\},$$
it follows that $D(a) = D_0(\l a + cR \r) = 0$. Hence
 $D(s) = D(a+c\sigma) = cD(\sigma) \in cD(S)$, which proves
 that $D(S)$ is ${{{{\it C}}}}$-divisible.  After replacing $D$ with $D|_S$, the lemma is proved.
\end{proof}

%\begin{rem} {\em The proof of the lemma shows that 
%$\wt{R}$ is isomorphic as an $R$-algebra to $\wt{A} \star I$, where $I$ is  the kernel of the canonical ring homomorphism
%$\phi:\wt{R} \rightarrow \wt{S}$.  We make use of this observation in the next section.}
%\end{rem}

With the aid of Lemma~\ref{pre splitting lemma}, we  now prove the abstract version of the 
correspondence theorem to be given in Theorem~\ref{DVR correspondence} of the next section.  Abbreviate $d_{S_C/A_C}$ by $d$, and for each ring $R$  between $A$ and $S$,  define an $S$-submodule of
$\Omega_{S_{{{{\it C}}}}/A_{{{{\it C}}}}}$ by:
$$\Omega(R) := \sum_{r \in R}Sd(r);$$
that is, $\Omega(R)$ is the $S$-submodule of $\Omega_{S_C/A_C}$  generated by the image of
$R$ under $d$.  
Our assumptions on $A$ and $S$ imply that $\Omega(S) = \Omega_{S_C/A_C}$.  For  since $d$ is $A_C$-linear, every element of $\Omega_{S_{{{{\it C}}}}/A_{{{{\it C}}}}}$
is of the form $\sum_{i=1}^n \frac{s_i}{c} \: d({\sigma_i})$ for
some $c \in {{{{\it C}}}}$ and $s_1,\ldots,s_n,$
$\sigma_1,\ldots,\sigma_n \in S$. Since $S/A$ is ${{{{\it
C}}}}$-divisible,  given an element in the above form, we may for
each $i =1,\ldots,n,$ write $\sigma_i = a_i + c\tau_i$ for some $a_i
\in A$ and $\tau_i \in S$. Thus, since $d$ is $A$-linear, we have
$$\sum_{i} \frac{s_i}{c} \: d({\sigma_i}) =
\sum_{i} \frac{s_i}{c} \: d(a_i +{c\tau_i}) = \sum_{i} {s_i} \:
d({\tau_i}) \in \Omega(S).$$  Therefore, $\Omega(S) =
\Omega_{S_{{{{\it C}}}}/A_{{{{\it C}}}}}$, and hence $S$ corresponds to $\Omega_{S_C/A_C}$ under the mappings in the following proposition.

%It asserts that the $C$-analytic sandwiches
%are in one-to-one correspondence with submodules $K$ of
%$\Omega(S)/K$ is $C$-torsion.

\begin{prop} \label{correspondence}
Let $A \subseteq S$
be an extension of rings, let $C$ be a multiplicatively closed subset of $A$ consisting of nonzerodivisors of $S$, and suppose  the inclusion mapping $A \rightarrow S$ is an analytic isomorphism along $C$.
  Then the 
  mappings $$R \mapsto \Omega(R) \: {\mbox{
and }}\: K \mapsto S \cap d_{S_C/A_C}^{-1}(K)$$ define a bijection between the intermediate 
rings 
 $A \subseteq R \subseteq S$  such that $S/R$ is $C$-torsion and $R \subseteq S$ is  a quadratic extension,  and the set of  $S$-submodules $K$ of
$\Omega(S)$ such that $\Omega(S)/K$ is $C$-torsion.
%
%and $K \mapsto R_K$ define a one-to-one correspondence between the
%set of $A$-subalgebras $R$ of $S$ such that  $S/R$ is ${\ff
%C}$-torsion and the set of $S$-submodules $K$ of $\Omega(S)$ such
%that $\Omega(S)/K$
%\Omega^*_{S_{\ff
%C}/A_{{{{\it C}}}}}$ such that $(\Omega^*_{S_{{{{\it C}}}}/A_{{{{\it C}}}}})/K$
%is
%${{{{\it C}}}}$-torsion.
\end{prop}

\begin{proof}   Let $d = d_{S_C/A_C}$.
First, we show these mapping are well-defined. Let $K$ be an $S$-submodule of $\Omega(S)$ such that $\Omega(S)/K$ is $C$-torsion, and let $R =S \cap  d^{-1}(K)$.  Then $R$ is
a ring, so we need only show that $R \subseteq S$ is a quadratic extension and $S/R$ is $C$-torsion.
 First, if $s \in
S$, then since $\Omega(S)/K$ is ${{{{\it C}}}}$-torsion, there
exists a nonzerodivisor $c \in {{{{\it C}}}}$ such that $cd(s) \in
K$. Since $c \in A$, then $d(cs) = cd(s) \in K$, so that $cs \in
S \cap d^{-1}(K)=R$. Hence $S/R$ is ${{{{\it C}}}}$-torsion.
Next, to prove that $R\subseteq S$ is a quadratic extension we show that (2) of  Lemma~\ref{start} can be substantiated by $R$ and $S$.  An appeal to Lemma~\ref{start} is possible here because $S/R$ is, as we have just seen, $C$-torsion, and since $S/R$ is an $A$-homomorphic image of the $C$-divisible $A$-module $S/A$, the $R$-module $S/R$ is $C$-divisible, and thus the hypotheses of Lemma~\ref{start} are satisfied.  Now, to verify (2) of the lemma,  define a mapping: $$D:S \rightarrow \Omega(S)/K:s \mapsto d(s) + K,$$ where $s \in S$.
Then since $d$ is a derivation, so is $D$ (note that $\Omega(S)/K$ is an $S$-module, since both $\Omega(S)$ and $K$ are).  Moreover: $$\Ker D = \{s \in S:d(s) \in K\} = S \cap  d^{-1}(K) = R.$$ Hence by Lemma~\ref{start}, $R \subseteq S$ is a quadratic extension.

Conversely, suppose that $A \subseteq R \subseteq S$, $R \subseteq S$ is a quadratic extension and $S/R$ is $C$-torsion.  
  We claim that $\Omega(S)/\Omega(R)$ is
$C$-torsion. Let $x \in \Omega(S)$ and  write $x =
\sum_{i=1}^ns_id(\sigma_i)$ for some $s_1,\ldots,s_n,$
$\sigma_1,\ldots,\sigma_n \in S$.  Since $S/R$ is $C$-torsion we may choose $c \in {{{{\it C}}}}$ such
that $c\sigma_i \in R$ for all $i =1,\ldots,n$. Then $cx =
\sum_{i=1}^ns_id(c\sigma_i)  \in \Omega(R),$ and this  proves that
$\Omega(S)/\Omega(R)$ is ${{{{\it C}}}}$-torsion. Therefore, the
mappings in the theorem are well-defined.

%Moreover, by Lemma~\ref{one-to-one lemma}, the mapping $K \mapsto
%d_S^{-1}$ is one-to-one.

%generated as an $S$-module by $d(S)$, it follows that $\Omega(S)/K$ is
%${{{{\it C}}}}$-torsion.

To see that $K \mapsto S \cap d^{-1}(K)$ is one-to-one, suppose that
$K$ and  $K'$ are  $S$-submodules of $\Omega(S)$ such that
$\Omega(S)/K$ and $\Omega(S)/K'$ are $C$-torsion and $K \ne K'$.  We show that $S \cap d^{-1}(K) \ne S \cap d^{-1}(K')$. Without loss of generality, suppose that there exists $k
\in K \setminus K'$, and write $k = \sum_{i=1}^n s_id(x_i)$ for
some $s_1,\ldots,s_n,x_1,\ldots,x_n \in S$.  Since $\Omega(S)/(K
\cap K')$ is ${{{{\it C}}}}$-torsion, there exists $c \in {{{{\it
C}}}}$ such that $cd(x_i) \in K \cap K'$ for all $i
=1,\ldots,n$. Also, since $S/A$ is ${{{{\it C}}}}$-divisible,
there exist for each $i$, $a_i \in A$ and $\sigma_i \in S$ such
that $s_i = a_i + c\sigma_i$. Therefore, we have: $$k = \sum_{i}
s_id(x_i) = d(\sum_{i} a_ix_i) + \sum_{i} \sigma_icd(x_i).$$
By the choice of $c$, it is the case that $\sum_{i}
\sigma_icd(x_i) \in K \cap K'$. Thus, since $k \not \in K'$, it
must be that $d(\sum_{i} a_ix_i) \not \in K'$. However,
$d(\sum_{i} a_ix_i) = k - \sum_{i} \sigma_i c d(x_i) \in K$, so
 $\sum_{i} a_ix_i$ is in  $S \cap d^{-1}(K)$ but not in $S \cap d^{-1}(K')$.  This proves that the mapping
  $K \mapsto S \cap d^{-1}(K)$ is one-to-one.

Next we show that the mapping $K \mapsto S \cap d^{-1}(K)$ maps onto the
set of rings $R$ such that $A \subseteq R \subseteq S$, $R \subseteq S$ is a quadratic extension and $S/R$ is $C$-torsion.  Suppose that $R$ is such a ring.
 We claim that there
exists an $S$-submodule $K$ of $\Omega(S)$ such that $\Omega(S)/K$
is $C$-torsion and $R =S \cap  d^{-1}(K)$. By Lemma~\ref{split
completion} there exists a ${{{{\it C}}}}$-linear derivation $D:S
\rightarrow L$ such that $L$ is a ${{{{\it C}}}}$-torsion-free
$S$-module; $D(S) = cD(S)$ for each $c \in {{{{\it C}}}}$; and  $R =
D^{-1}(N)$ for some $S$-submodule $N$ of $L$ such that $L/N$ is
$C$-torsion. By extending $D$ via the quotient rule, we may  in fact view $D$ as a
derivation from $S_{{{{\it C}}}}$ to $L_{{{{\it C}}}}$, so
that $R = S \cap D^{-1}(N)$. Now, as we have shown above, $\Omega(S)
= \Omega_{S_{{{{\it C}}}}/A_{{{{\it C}}}}}$, so by the universal
property of the module of K\"ahler differentials (2.1),
there exists an $S$-module homomorphism $f:\Omega(S) \rightarrow
L_C$ such that $D = f \circ d$. Let $K = f^{-1}(N)$.  We
claim that $R = S \cap d^{-1}(K)$. Indeed, suppose that $s \in S$ with
$d(s) \in K$. Then $D(s) = f(d(s)) \in f(K) \subseteq N$, so
that $s \in S \cap D^{-1}(N) = R$. Therefore, $S \cap d^{-1}(K) \subseteq
R$.  The reverse inclusion is clear, since $D(R) \subseteq N$ and $D
= f \circ d$ imply that $d(R) \subseteq f^{-1}(N) = K$.
Hence $R = S \cap d^{-1}(K)$.  Finally, to see that $\Omega(S)/K$ is
$C$-torsion, let $y \in \Omega(S)$.  Then $f(y) \in L_C$, so since
$L/N$ is $C$-torsion, there exists $c \in C$ such that $cf(y) \in
N$, and hence $cy \in f^{-1}(N) = K$.
%and it remains to show that $\Omega(S)/K$ is
%${{{{\it C}}}}$-torsion.  Let $x \in \Omega(S)$. Then $x =
%\sum_{i=1}^ns_id(\sigma_i)$ for some
%$s_1,\ldots,s_n,\sigma_1,\ldots,\sigma_n \in S$.  Choose $c \in
%{{{{\it C}}}}$ such that $c\sigma_i \in R$ for all $i =1,\ldots,n$.
%Then
%$$f(cx) = f(\sum_{i=1}^ns_id_S(c\sigma_i)) = \sum_{i=1}^n
%s_if(d_S(c\sigma_i)) = \sum_{i=1}^n s_iD(c\sigma_i) \in N.$$  Thus
%$cx \in f^{-1}(N) =K$, which proves that $\Omega(S)/K$ is ${{{{\it
%C}}}}$-torsion, and this completes the proof.
\end{proof}

The module $\Omega(R)$ plays a role similar to that of the module of K\"ahler differentials in that it shares some universal properties with this module:

\begin{lem} \label{pre-construction u} Let $A \subseteq S$
be an extension of rings, let $C$ be a multiplicatively closed subset of $A$ consisting of nonzerodivisors of $S$, and suppose  the inclusion mapping $A \rightarrow S$ is an analytic isomorphism along $C$. Let $L$ be an $S_C$-module, 
 and suppose that 
there is an $A_C$-linear derivation $D:S_C \rightarrow L$ such that 
 $D(S_C)$ generates $L$ as an $S_C$-module. There exists a surjective $S_C$-module homomorphism $\alpha:\Omega_{S_C/A_C} \rightarrow L$ such that $D = \alpha \circ d_{S_C/A_C}$ and for 
   each
 $S$-submodule $K$ of $L$ with $K_C = L$, 
     then with $R  = S \cap D^{-1}(K)$, it is the case that  $\alpha(\Omega(R)) = K$, $\Omega(R) = \alpha^{-1}(K)$ and $\alpha$ restricted to $\Omega(R)$  is an analytic isomorphism along $C$.  
\end{lem}

\begin{proof}    Let $d = d_{S_C/A_C}$.
 Since $D:S_C \rightarrow L$ is an $A_C$-linear derivation and (as noted earlier)  $\Omega(S) = \Omega_{S_C/A_C}$, then there
exists an $S_C$-module homomorphism $\alpha:\Omega(S) \rightarrow L$
such that $D = \alpha \circ d$; see (2.1). Since the image of $D$ generates $L$, it follows that $\alpha$ is surjective. 
%
% That $\alpha$ is surjective will be established in the course of the proof. 
Let $K$ be an  
$S$-submodule of $L$ with $K_C = L$, and let 
    $R  = S \cap D^{-1}(K)$.
  We claim first that $\Omega(R) =
\alpha^{-1}(K)$. Let $x \in \Omega(R)$, and write $x =
\sum_{i=1}^ns_id(r_i)$, where $s_1,\ldots,s_n \in S$ and
$r_1,\ldots,r_n \in R$.  Then $$\alpha(x) = \sum_i s_i
\alpha(d(r_i)) = \sum_i s_iD(r_i) \in K,$$ and this shows that
$\Omega(R) \subseteq \alpha^{-1}(K)$. To see that the reverse
inclusion holds, let $y \in \alpha^{-1}(K)$, and write $y =
\sum_{i=1}^n\sigma_id(\tau_i)$ for some $\sigma_1,\ldots,\sigma_n,
\tau_1,\ldots,\tau_n \in S$.  Since $S/R$ is $C$-torsion, there
exists $c \in C$ such that $c\tau_i \in R$ for all $i$.  Also, since
$S/A$ is $C$-divisible, for each $i$ there exist $a_i \in A$ and
$u_i \in S$ such that $\sigma_i = a_i + cu_i$.
 Then: $$y =
\sum_i \sigma_id(\tau_i) = d(\sum_i a_i \tau_i) + \sum_i
u_id(c\tau_i).$$  By assumption $\alpha(y) \in K$, so, since $D =
\alpha \circ d$, we have: $$D(\sum_i a_i \tau_i) + \sum_i u_i
D(c\tau_i) \in K.$$  Now, since for each $i$, $c\tau_i \in R =
D^{-1}(K) \cap S$, it follows that $\sum_i u_iD(c\tau_i) \in K$.
Therefore, $D(\sum_i a_i \tau_i) \in K$, which shows that $\sum_i
a_i \tau_i \in S \cap D^{-1}(K) = R$.  Thus $$y = d(\sum_i a_i
\tau_i) + \sum_i u_id(c\tau_i) \in \Omega(R),$$ and this proves
that $\Omega(R) = \alpha^{-1}(K)$.
%  Next we show that $\alpha$ maps $\Omega(S)$ onto $K_C$.
%   To this end, we first claim that $K_C$ is generated as an
%   $S$-module by $D(S)$.  Since $D$ is $C$-linear and $K_C$ is generated as an $S$-module by $D(S_C)$,  every element $x$ of $K_C$ is
%   of the form $\sum_{i} \frac{s_i}{c}D(\sigma_i)$, where $s_i,
%   \sigma_i \in S$ and $c \in C$.  Since $S/A$ is $C$-divisible, for each $i$,
%    $\sigma_i = a_i +
%   c\tau_i$ for some $a_i \in A$ and $\tau_i \in S$.  Hence
%   $D(\sigma_i) = D(a_i + c\tau_i) = cD(\tau_i)$.  Consequently, $x
%   = \sum_i \frac{s_i}{c}D(\sigma_i) = \sum_i {s_i}D(\tau_i)$, and therefore,
%   $K_C$ is generated as an $S$-module by $D(S)$. Then
%   $\alpha$ maps $\Omega(S)$ onto $K_C$
%  since:
%$$K_C = \sum_{s \in S}SD(s) = \sum_{s \in S}S\alpha(d(s)) =
%\alpha(\sum_{s \in S}Sd(s)) = \alpha(\Omega(S)).$$  
Moreover, from the fact that $\alpha(\Omega(S)) = K_C$ and $\alpha^{-1}(K) = \Omega(R)$,  it follows that $\alpha(\Omega(R)) =
K$. 

Finally, to prove that $\alpha$ is an analytic isomorphism, note first that for each $c \in C$, $\alpha$ induces the mapping:
$$\alpha_c:\Omega(R)/c\Omega(R)
\rightarrow K/cK:x + c\Omega(R) \mapsto \alpha(x) + cK,$$ where  $x \in K$.  This mapping is clearly a well-defined $S$-module
homomorphism, and since $\alpha$ is onto, so is $\alpha_c$.  To see
that it is injective, suppose $x \in \Omega(R)$ such that
 $\alpha(x) \in cK$.  Since $x \in \Omega(R)$, there exist $s_1,\ldots,s_n \in S$ and $r_1,\ldots,r_n \in R$ such that $x = \sum_i s_id(r_i)$.  For each $i$, write $s_i = a_i + c\sigma_i$, where $a_i \in A$ and $\sigma_i \in S$.  Then $$x = d(\sum_i a_ir_i) + c\sum_i \sigma_i d(r_i).$$  Set $r = \sum_i a_ir_i$ and $y = \sum_i \sigma_i d(r_i)$, so that $x = d(r) + cy$, with $y \in \Omega(R)$.  Then using the fact
that $D = \alpha \circ d$, it follows that $ D(r) + c\alpha(y) =
\alpha(x) \in cK$. Since $y \in \Omega(R)$ and $\alpha(\Omega(R)) =
K$, we have $c\alpha(y) \in cK$.  Hence $D(r) \in cK$. Now
since $S/A$ is $C$-divisible, we may write $r = a+cs$ for some $a
\in A$ and $s \in S$.  Thus, since $D$ is $A$-linear, $cD(s) = D(cs)
= D(a+cs) = D(r) \in cK$, and since $L=K_C$ is $C$-torsion-free, $D(s)
\in K$. But then $s \in S \cap D^{-1}(K) = R$, so $d(r) = d(a +
cs) = cd(s) \in c\Omega(R)$.  Therefore, $x  = d(r) + cy \in
c\Omega(R)$, and this proves that $\alpha_c$ is injective.  Hence
$\alpha_c$ is an isomorphism of $S$-modules. \end{proof}

%\begin{lem}  Let $A \subseteq R \subseteq S$ be a $C$-analytic sandwich of
%domains, where $C$ is a set of nonzerodivisors of $S$ contained in
%$A$.   Let $P$ be a prime ideal of $R$.  Then $R_{P} = S_{P}$ if and
%only if $\Omega(S) = \Omega(R)_P$. Also, if $P'$ is a prime ideal of
%$S$ that meets $C$, then $R_{P' \cap R} = S_P$ if and only if
%$\Omega(S) = \Omega(R)_{P'}$.
%\end{lem}

The next proposition will be useful in Section 4 for describing the completions of bad stable rings in immediate extensions of DVRs.   

\begin{prop} \label{pre-construction} Let $A \subseteq S$
be an extension of rings, let $C$ be a multiplicatively closed subset of $A$ consisting of nonzerodivisors of $S$, and suppose  the inclusion mapping $A \rightarrow S$ is an analytic isomorphism along $C$. Let  $K$ be a $C$-torsion-free $S$-module, and suppose that 
there is an $A_C$-linear derivation $D:S_C \rightarrow K_C$ such that 
 $D(S_C)$ generates $K_C$ as an $S_C$-module.  Let   $R  = S \cap D^{-1}(K)$.  
 Then the mapping $$f:R \rightarrow S \star K:r \mapsto (r,D(r))$$
 is an analytic isomorphism along $C$
\end{prop}

\begin{proof}  
%(2)   As in the proof of Theorem~\ref{correspondence},
%it follows
%that $K_C/K$ is $C$-torsion; and conversely, if $K_C/K$ is $C$-torsion,
%then $D^{-1}(K)$ is a $C$-full subring of $S$.
%These  mappings are well-defined, and by K_Cemma~\ref{one-to-one
%lemma}, the mapping $K \mapsto D^{-1}(K)$ is injective.
% To prove
%that this mapping is onto, we observe first that by ?, we may assume
%that $D$ is in fact a restriction to $S$ of a $C$-linear derivation
%$D:S_C \rightarrow K_C_C$.  Thus, since by Proposition~\ref{full
%equal}, $\Omega(S) = \Omega_{S_C/A_C}$, there is by (?) a mapping
%$\alpha:\Omega(S) \rightarrow K_C$ such that $D = \alpha \circ
%d_{S_c/A_c}$.
  We prove the proposition by verifying three claims.
  \smallskip
  
 {\textsc{Claim 1:}}   For each $c \in {{{{\it C}}}}$,
%$S = D^{-1}(cK) + cS$,
$K = D(cS \cap R) + cK$.
\smallskip

Let $c \in {{{{\it C}}}}$, and let $k \in K$.
Then
 by Lemma~\ref{pre-construction u}, 
 there exists a surjective $S$-module homomorphism $\alpha:\Omega_{S_C/A_C} \rightarrow K_C$ such that $D = \alpha \circ d_{S_C/A_C}$ and $\alpha(\Omega(R)) = K$,
 Thus  
 $k = \sum_{i=1}^n s_iD(r_i)$, for some $r_i \in R$ and $s_i \in
S$.  Since $S/A$ is $C$-divisible, 
for each $i$, we may write $s_i = a_i + c\sigma_i$ for some $a_i
\in A$ and $\sigma_i \in S$. Thus, setting $r =\sum_{i} a_ir_i$, we
have $k = D(r) + \sum_{i} c\sigma_iD(r_i).$ Similarly, we may write
$r = a + c\sigma$ for some $a \in A$ and $\sigma \in S$, so that
$c\sigma \in cS \cap R$. Thus, since $D$ is $A$-linear, $D(r) =
D(c\sigma)$, and we have that $k = D(c\sigma) + \sum_{i}
c\sigma_iD(r_i)$. Hence $k \in D(cS \cap R) + c K$, which verifies
that $K = D(cS \cap R) + c K$.

\smallskip
{\textsc{Claim 2:}} For each $c \in {{{{\it C}}}}$,
 $cS \cap D^{-1}(cK) = cR$.
\smallskip

 Let $c \in {{{{\it
C}}}}$ and $s \in S$ such that $D(cs) \in c K$. Then $cD(s) = D(cs)
 \in c K$. Hence, since $K$ is ${{{{\it C}}}}$-torsion-free, $D(s)
\in K$. But $R = S \cap D^{-1}({K})$, so $s \in R$. This shows that
$cS \cap D^{-1}(c K) = cR$.

\smallskip
{\textsc{Claim 3:}} $f$ is an analytic isomorphism along $C$.  
\smallskip

It is enough to
show that the cokernel of $f$ is $C$-torsion-free and
$C$-divisible. First we show that the cokernel of $f$ is ${{{{\it
C}}}}$-torsion-free. Suppose that $s \in S$, $k \in K$ and $c \in
{{{{\it C}}}}$ such that $c \cdot (s,k) = (r,D(r))$ for some $r \in
R$. Then $D(cs) \in cK$, so that $cs \in cS \cap D^{-1}(cK)$. Hence
by Claim 2 and the fact that $c$ is a nonzerodivisor in $S$, $s \in R$.
Also, since $cD(s) = D(cs) = ck$, and $K$ is ${{{{\it
C}}}}$-torsion-free, we have $D(s) = k$, which proves that $(s,k) = (s,D(s))
\in f(R)$.

Next we show that the cokernel of $f$ is ${{{{\it
C}}}}$-divisible. Let $c \in C$ and $k \in K$.  We claim that $(0,k)
\in f(R) + c \cdot (S \star K)$. By Claim 1 we may write $k =
D(c\sigma) + ck'$ for some $k' \in K$ and $\sigma \in S$  such that
$c\sigma \in R$. Now $$(0,k) = (c\sigma,D(c\sigma)) + c \cdot
(-\sigma ,k') \in f(R) + c \cdot (S \star K).$$
Thus to complete the proof that the cokernel of $f$ is ${{{{\it
C}}}}$-divisible, it suffices to show that for each $s \in S$,  $(s,0)  \in
f(R) +
 c \cdot (S\star
K)$.
%Now for each $r \in R$, we have by (3), $D(r) - D(cs) \in cK$ for
%some $s \in S$ such that $cs \in R$.  Thus $r-cs \in D^{-1}(cK)$,
%and $R = cS \cap R + D^{-1}(cK)$.
%To this end, we claim first that $S = R + cS$. Let $s \in S$. Then since $S/R$ is $C$-torsion,
%there exists $d \in {{{{\it C}}}}$ such that $ds \in R$. Hence
%$D(ds) \in K = D(cdS \cap R) + cdK$, and we may write $D(ds) =
%D(cd\sigma)+cdk$ for some $\sigma \in S$ and $k\in K$ such that $cd\sigma \in R$. Thus $D(ds -
%cd\sigma) \in cdK$, so that using Claim 1, we have:
% $$ds - cd\sigma \in D^{-1}(cdK) \cap dS
%\subseteq D^{-1}(dK) \cap dS = dR.$$
%Thus since $d$ is a nonzerodivisor in $S$,
%it must be that $s - c\sigma \in R$. Consequently, $s \in R +
%cS$, and we have shown that $S = R + cS$.
 % 
%
%Since $S = R + cS$, it
%suffices to show that $R \subseteq D^{-1}(cK)+cS$.  Let $r \in R$.
%By Claim 1, $D(r) - D(c\sigma) \in cK$ for some $\sigma \in S$ such that
%$c\sigma \in R$.  Thus $D(r-c\sigma) \in cK$, so that $r-c\sigma \in
%D^{-1}(cK)$. Therefore, $R \subseteq D^{-1}(cK) + cS$, and we have
%$S = D^{-1}(cK) + cS$, as claimed.
%
Let $s \in S$.  
Since $S/A$ is $C$-divisible, then $S = A + cS$, and  we may
write $s = a+ c\sigma$, where $\sigma \in S$ and $a \in A$. Therefore, since $D(a) = 0$, we have  $$(s,0) = (a,D(a)) + (c\sigma,0)  \in f(R) +
c \cdot (S \star K),$$ and 
%since we have shown above that for all $k
%\in K$, $(0,k) \in f(R) + c \cdot (S \star K)$,  
we conclude that
the cokernel of $f$ is ${{{{\it C}}}}$-divisible.
\end{proof}

%and by  Lemmas~\ref{lifts} and~\ref{f basic}, $f_D$ lifts to an
%isomorphism from $\wt{R}$ to the ${{{{\it C}}}}$-completion of $S
%\star K$, we have
% by Proposition~\ref{basics} and
%the $A$-linearity of $D$,
%$$R/IR \cong f_D(R)/I \cdot f_D(R) \cong (S \star K)/I \cdot (S \star
%K) \cong S/IS \star K/IK.$$ This proves the lemma.

\begin{cor}  \label{analytic lift} Let $A \subseteq R \subseteq  S$
be an extension of rings, let $C$ be a multiplicatively closed subset of $A$ consisting of nonzerodivisors of $S$, and suppose  the inclusion mapping $A \rightarrow S$ is an analytic isomorphism along $C$.
If  $R \subseteq S$ is a quadratic extension and $S/R$ is $C$-torsion, then the mapping $$f:R \rightarrow S \star \Omega(R):r \mapsto (r,d(r))$$ is an analytic isomorphism along $C$.
\end{cor}

\begin{proof}  The $S$-module $\Omega(R)$ is generated by $d_{S_C/A_C}(R)$, and by Proposition~\ref{correspondence}, $R = S \cap d^{-1}_{S_C/A_C}(\Omega(R))$, so the corollary follows from Proposition~\ref{pre-construction}
\end{proof}

\label{4:analytic sandwiches and stable ideals}

\section{Stable rings between DVRs}

\label{next section}

We apply the results of the previous section on analytic isomorphisms to our main case of interest, where $A \subseteq S$ is an immediate extension of DVRs.  It is straightforward to check that the inclusion mapping of an immediate extension $U \subseteq V$ of DVRs is an analytic isomorphism along the multiplicatively closed set $C = U \setminus \{0\}$ and that $U_C$ and $V_C$ are the quotient fields of $U$ and $V$, respectively, so all of the results of the previous section can be translated into the present context.   
Let $Q$ and $F$ denote the quotient fields of $U$ and $V$, respectively.  
We consider rings $R$ between $U$ and $V$, and as in the last section we  associate to each such ring the $V$-submodule  $\Omega(R)$ of $\Omega_{F/Q}$ generated by $d_{F/Q}(R)$; that is, $$\Omega(R):= \sum_{r \in R}Vd_{F/Q}(r).$$ 
The correspondence in Proposition~\ref{correspondence} translates now into the form given in Theorem~\ref{DVR correspondence}. Recall that when $L$ is a torsion-free module over a domain, then a submodule $K$ of $L$ is {\it full} if $L/K$ is a torsion module.

\begin{thm} \label{DVR correspondence} 
Let  $U \subseteq V$ be an immediate extension of DVRs having quotient fields $Q$ and $F$, respectively.    
There is a one-to-one correspondence between bad stable rings $R$ with $U \subseteq R \subseteq V$ and normalization $V$, and  proper full  $V$-submodules $K$ of $\Omega_{F/Q}$ given by $$R \mapsto \Omega(R) {\mbox{ and }} K \mapsto V \cap d_{F/Q}^{-1}(K).$$  %Moreover, each such stable ring $R$ arising this way is either equal to $V$ or is a bad stable ring.  
  
\end{thm}  

\begin{proof}  In light of Proposition~\ref{correspondence}, with $C = U \setminus \{0\}$, all that needs to be verified for the correspondence is that  a ring $R$ between $U$ and $V$ is stable with normalization $V$ if and only if $R \subseteq V$ is a quadratic extension and  $V/R$ is a torsion $U$-module.  
If $R$ is stable with normalization $V$,  then $V/R$ is a torsion $R$-module.  Since $R$ has dimension one and dominates $U$, it follows that $QR = F$, and hence  $V/R$ is torsion not just as an $R$-module, but as a $U$-module also.  
  Moreover, by (2.4), $R \subseteq  V$ is a quadratic extension.  Conversely, if $V/R$ is a torsion $U$-module and $R \subseteq V$ is a quadratic extension, then $R$ and $V$ share the same quotient field $F$, and since $R \subseteq V$ is an integral  extension and $V$ is a DVR, it follows that $V$ is the normalization of $R$. Hence by (2.4), $R$ is either equal to $V$ or $R$ is a bad stable ring.  
\end{proof}

The theorem makes it possible to pinpoint when there is an analytically ramified local Noetherian $U$-subalgebra of $V$ having normalization $V$.  To prove this we recall a fact due to  Matlis  that every one-dimensional analytically ramified local Noetherian domain has a bad $2$-generator  overring \cite[Theorem 14.16]{M1}.  This result is proved from a different point of view in \cite[Theorem 5.10]{OlbGFF}.  In fact, for the purposes of the corollary,   Theorem~\ref{tightly} below would also suffice.

\begin{cor} \label{main theorem 1} Let $U \subseteq V$ be an immediate extension of DVRs with quotient fields $Q$ and $F$, respectively.  Then the following are equivalent.  

\begin{itemize} \item[(1)]  There exists an analytically ramified local Noetherian domain containing $U$ and having normalization $V$.  

\item[(2)]   Either  $F$ has characteristic $0$ and $F/Q$ is not algebraic, or  $F$ has characteristic $p>0$ and $F \ne Q[F^p]$.  
\end{itemize}
\end{cor}   

\begin{proof}   
If there exists an analytically ramified local Noetherian domain $A$ containing $U$ and having normalization $V$, then as discussed before the corollary, there exists a bad stable ring $R$  containing $A$ and having normalization $V$.  Thus by Theorem~\ref{DVR correspondence}, $\Omega_{F/Q} \ne 0$.  By (2.2), the nonvanishing of $\Omega_{F/Q}$ is equivalent to (2) being valid for $F/Q$.  Therefore, by Theorem~\ref{DVR correspondence}, to prove the converse it suffices to note that the nonzero $F$-vector space $\Omega_{F/Q}$ contains a proper full $V$-submodule. 
\end{proof}

\begin{rem} {\em The idea of constructing analytically ramified local Noetherian domains using derivations originates with Ferrand and Raynaud in \cite{FR}.  Their method, along with a generalization by Goodearl and Lenagan \cite{GL}, motivates Theorem~\ref{DVR correspondence}, although our approach and results differ.  The connection with the articles \cite{FR} and \cite{GL} is more evident in \cite{OlbFR}.}
\end{rem}

Also to $R$ we associate the following  cardinal number, which did not have an analogue in the last section:  $$\epsilon_R: = \dim_{V/{\ff M}_V} \Omega(R)/{\ff M}_V\Omega(R).$$  We see in the next theorem that this cardinal is one less than the embedding dimension of $R$.  In the theorem, 
%This cardinal places a lower bound on the dimension of the $F$-vector space  $\Omega_{F/Q}$:  
 we deduce from Corollary~\ref{analytic lift} a description of the completions of bad stable rings between $U$ and $V$.  Recall from (2.5) that the $M$-adic  completion $\widehat{R}$ of a bad stable ring $R$ has a prime ideal $P$ such that $P^2 = 0$ and $\widehat{R}/P$ is a DVR.  Theorem~\ref{lifting theorem} shows that $P$ splits $\widehat{R}$, making $\widehat{R}$ an idealization.  
 In the theorem, $\widehat{R}$ denotes the completion of $R$ in the $M$-adic topology, where $M$ is the maximal ideal of $R$, while  $\widehat{V}$ and $\widehat{\Omega}(R)$ denote    the completions of $V$ and $\Omega(R)$ in the ${\ff M}_V$-adic topology (which since $M$ is stable and $V$ is a DVR  is easily seen to coincide with the $M$-adic topology).

\begin{thm} \label{stable theorem} \label{stable ed} 
\label{lifting theorem} 
Let  $U \subseteq V$ be an immediate extension of DVRs having quotient fields $Q$ and $F$, respectively, and let $R$ be a  stable ring between $U$ and $V$ with   normalization  $V$.  Then the following statements hold for $R$.  

\begin{itemize}
 \item[(1)]   The mapping $f$ defined by $$f:R \rightarrow V \star \Omega(R):r \mapsto (r,d_{F/Q}(r))$$ lifts to an isomorphism of rings: 
  $$\widehat{R} \rightarrow \widehat{V} \star \widehat{\Omega}(R).$$  

 \item[(2)]  $R$ is  Noetherian  (with embedding dimension
$\epsilon_R + 1$) if and only if $\epsilon_R$ is finite.
 
\item[(3)]  The dimension of the $F$-vector space $\Omega_{F/Q}$ is at least $\epsilon_R$.  

%has cardinality 
%a lower bound for the the $ $\epsilon_R \leq \dim_F\Omega_{F/Q}$.  
  %If also $\epsilon_R$ is finite, then  $\widehat{\Omega}(R)$  is a free $\widehat{V}$-module of rank $\epsilon_R$.  
\end{itemize}
\end{thm}

\begin{proof} (1) Observe that $V/R$ is $C$-torsion for $C = U \smallsetminus \{0\}$ since $F = QV = QR$.  Thus by Corollary~\ref{analytic lift}, for each $0 \ne c\in U$, the induced mapping $$f_c: R/cR \rightarrow V/cV \star \Omega(R)/c\Omega(R)$$ is an isomorphism of $U$-algebras.  A straightforward verification using the fact that $d_{F/Q}$ is $U$-linear  shows then that the mapping $f_c$ induces an isomorphism of rings: $$\varprojlim R/cR \rightarrow (\varprojlim V/cV) \star (\varprojlim \Omega(R)/c\Omega(R)),$$ where the inverse limits range over all $0 \ne c \in U$. Since $R$ is stable, the maximal ideal $M$ of $R$ has the property that $M^2 = mM$ for some $m \in M$.  Since $R$ is one-dimensional, then $\sqrt{mR} = \sqrt{cR} = M$ for any choice of $0 \ne c \in {\ff M}_U$.  Moreover, since $M^2 \subseteq mR \subseteq M$, it follows that for each $0 \ne c \in {\ff M}_U$, there is a power of $M$ contained in $cR$.  Therefore, $\widehat{R}$ is isomorphic to $\varprojlim R/cR$.  Also, since $U \subseteq V$ is an immediate extension, it follows that $\widehat{V}$ is isomorphic to $\varprojlim V/cV$ and $\widehat{\Omega}(R)$ is isomorphic to $\varprojlim \Omega(R)/c\Omega(R)$.    
 Therefore, since all the maps involved are natural,  $f$ lifts to the isomorphism in the theorem. 

(2)
 Let $K = \Omega(R)$.  
By Corollary~\ref{analytic lift}, the mapping $f_c:R/cR \rightarrow V/cV \star K/cK$ is an isomorphism for each $0 \ne c \in U$.  Since $R$ has dimension $1$ and the maximal ideal of $R$ contains ${\ff M}_U$,  every principal ideal of $R$ contains a power of some nonzero element of $U$.  From this 
 it follows that $R$ is a Noetherian domain if and only if  for each $0 \ne c \in U$,   $K/cK$ is a finitely generated $V$-module.  
Thus if $R$ is a Noetherian domain, it must be that $K/{\ff M}_VK$  is a finitely generated  $V$-module, and
hence its dimension $\epsilon_R$ as a $V/{\ff M}_V$-vector space is finite.
Conversely, if $K/{\ff M}_VK$ has finite dimension as a $V/{\ff M}_V$-vector space  and $0
\ne c \in {\ff M}_U$, then since ${\ff M}_V^k = cV$ for some $k>0$, it follows that
$K/cK$ is a finitely generated  $V$-module.     Thus if $\epsilon_R$ is
finite, then  $R$ is Noetherian.   
The assertion about embedding dimension now follows from (1) and the fact that finitely generated modules over DVRs are free.  For by the theorem, $\widehat{R} \cong \widehat{V} \star \widehat{\Omega}(R)$, 
and when $R$ is Noetherian, so is $\widehat{R}$, so that $\widehat{\Omega}(R)$ is a finitely generated free $\widehat{V}$-module.  Its rank is $\epsilon_R$, so the embedding dimension of $\widehat{R}$ is one more than this rank.

(3) Let ${\ff M} = {\ff M}_V$, $K = \Omega_{F/Q}$ and $d = d_{F/Q}$.  A consequence of  Claim 1 in the proof of Theorem~\ref{pre-construction} is that   
 $K = d(R) + {\ff M}K$.  
Thus we may choose a collection ${\cal B}$ of elements of $R$ such that
$\{d(r):r \in {\cal B}\}$ consists of  representatives of basis
elements of the $V/{\ff M}$-vector space $K/{\ff M}K$.  To prove (3) then, it
suffices to show that the elements $d(r)$, $r \in {\cal B}$, are linearly
independent elements of the $F$-vector space $\Omega_{F/Q}$.  To
this end, suppose that $r_1,\ldots,r_m \in {\cal B}$ and $s_1,\ldots,s_m \in F$ such that
$s_1d(r_1) + \cdots + s_m d(r_m) = 0$.  By
clearing denominators, we may assume that each $s_i \in V$.
Moreover, if each $s_i \in {\ff M}$, then we may
divide by the generator of the principal ideal ${\ff M}$ of the DVR $V$ until some $s_i$,
say $s_1$, is in $V \setminus {\ff M}$.  But then  $s_1 + {\ff M}$ is a nonzero
coefficient of $d(r_1) + {\ff M}K$ in the equation
$$\sum_{i=1}^{m} (s_i+{\ff M})(d(r_i)+{\ff M}K) = 0 + {\ff M}K,$$
 which forces the images of $d(r_1),\ldots,d(r_m)$ to be
linearly dependent in the $V/{\ff M}$-vector space $K/{\ff M}K$.  This
contradiction proves that $\epsilon_R \leq \dim_F \Omega_{F/Q}$.  
\end{proof}

%\item[(6)]  If $R$ is a Noetherian domain,  the lattice of rings between $R$
%and $S$    is isomorphic to the lattice of $R$-submodules of
%$\bigoplus_{i=1}^{\epsilon_R} F/S$.

\section{Local rings  between DVRs}

\label{4:application: stable rings in function fields}

In the last section we considered stable rings in an immediate extension of DVRs.   In this section, we consider
extensions $U \subseteq A \subseteq V$ where $U \subseteq V$ is an immediate extension of DVRs and $A$ is a local Noetherian domain birationally dominated by $V$, and in this more nuanced setting we describe the bad stable rings between $A$ and $V$.
To illustrate how such a setting can occur, consider a local Noetherian domain $A$ containing a field $k$ such that $A = k + {\ff m}$, where ${\ff m}$ is the maximal ideal of $A$.  Then straightforward arguments show that there exist DVRs $U$ and $V$ such that $U \subseteq A \subseteq V$ and $V$ birationally dominates $A$ if and only if there is an embedding $f:A \rightarrow k[[T]]$, where $T$ is an indeterminate for $k$, such that ${\ff m}k[[T]] = Tk[[T]]$.  Geometrically, the existence of such an embedding $f$ can be expressed as saying that there is a fat, nonsingular $k$-arc $\Spec(k[[T]]) \rightarrow \Spec(A)$; see for example \cite{Ishii}.

%A specific instance of such a setting can be de

% in the context of analytic arcs.  Let $X$ be an irreducible variety over a field $k$ with function field $F$.  A morphism of schemes $\alpha:\Spec(k[[T]]) \rightarrow X$ is  a  {\it $k$-rational analytic arc}.  Let $x$ be the image in $X$ of the  closed point of $\Spec(k[[T]])$, and let $A$ be the local ring of $X$ at $x$.  The arc $\alpha$ is  {\it fat} if the corresponding ring homomorphism $\alpha:A \rightarrow k[[T]]$ is an embedding; $\alpha$ is {\it nonsingular} if there exists $z \in A$ such that $zk[[T]] = Tk[[T]]$; see for example \cite{Ishii}.  Thus when $\alpha$ is a nonsingular fat arc and  $z$ is as above, then setting $U = k[z]_{(z)}$ and $V = F \cap k[[T]]$, then $U \subseteq A \subseteq V$ and $U \subseteq V$ is an immediate extension of DVRs.        

  Returning to the abstract setting $U \subseteq A \subseteq V$, 
recall that a quasilocal domain $B$ {\it birationally dominates} $A$ if $A \subseteq B$, $A$ and $B$ have the same quotient field and the maximal ideal of $A$ is contained in the maximal ideal of $B$.  
 Using the subring $U$ of $A$, we describe the smallest bad stable ring between $A$ and $V$.  This provides a somewhat natural source of one-dimensional analytically ramifed  Noetherian domains birationally dominating local Noetherian domains.  %   Our standing assumptions for this section are:

%\begin{quote} {\it $U \subseteq V$ is an immediate extension of DVRs having quotient fields $Q$ and $F$, respectively, with $Q \subseteq F$;  $A$ is a local Noetherian domain with maximal ideal ${\ff m}$   such  that   $U \subseteq A \subsetneq V$ and $V$ birationally dominates $A$; and $R = \Ker d_{V/A}$. } 
%\end{quote}

Our focus in this section is on a specific ring between $A$ and $V$, the ring given by $R = \Ker d_{V/A}$, where $d_{V/A}:V \rightarrow \Omega_{V/A}$ is the exterior differential of the ring extension $V/A$.    
Since $d_{V/A}$ is a derivation,  $R = \Ker d_{V/A}$ is a ring between $A$ and $V$.    Our interest in this ring lies in the fact that it is the smallest bad stable ring between $A$ and $V$.  More precisely,  
in \cite[Theorem 5.9]{OlbGFF} it is shown that if $B$ is a local Noetherian domain with maximal ideal ${\ff m}$ and $W$ is a DVR which birationally dominates $B$ and has the property that $W = B +{\ff m}W$ (in which case we say that $W$ {\it tightly dominates} $B$), then the ring $\Ker d_{W/B}$ is a bad stable ring that is contained in every stable ring between  $B$ and $W$; equivalently, this kernel is the smallest ring between $B$ and $W$ which forms a quadratic extension with $W$.  It is not clear whether in the general context just mentioned for $B$ the kernel must always be a Noetherian ring.  However, in the setting of this article, it is not only Noetherian, but its maximal ideal is extended from that of the base ring:

\begin{thm} \label{tightly} 
Let $U \subseteq A \subsetneq  V$ be an extension of local Noetherian domains, where $U \subseteq V$ is  an immediate extension of DVRs  and $V$ birationally dominates $A$.
  The ring  $R = \Ker d_{V/A}$
 is a bad Noetherian stable ring that tightly dominates $A$; $R$ has normalization $V$; $R$ has   maximal ideal ${\ff m}R$ extended from the maximal ideal ${\ff m}$ of $A$;  and  $R$   is contained in every stable ring between $A$ and $V$.   
\end{thm}

\begin{proof}   
Observe that ${\ff m} = {\ff M}_V \cap A$, since $V$ dominates $A$.  Since $U \subseteq V$ is an immediate extension, 
 ${\ff M}_U = {\ff M}_V \cap U$, and hence ${\ff M}_U = {\ff m} \cap U$.  The  maximal ideal of $U$ extends to the maximal ideal of $V$, and hence ${\ff M}_V = {\ff m}V$.  Also from the immediacy of the extension, $V = U + {\ff m}V$, so that necessarily $V = A + {\ff m}V$.  Therefore, $V$ tightly dominates $A$, and hence as noted before the theorem, $R$ is a bad stable domain with normalization $V$ that is contained in every stable ring between $A$ and $V$.  Thus it remains to show that $R$ tightly dominates $A$ and $R$ is  Noetherian with maximal ideal ${\ff m}R$.  Since $R$ has Krull dimension $1$ and ${\ff m}$ is a finitely generated ideal of $A$, it  suffices to prove that $R = A + {\ff m}R$.      
But from $V = A +{\ff m}V$ we deduce that $R = A + ({\ff m}V \cap R)$, and hence the residue field of $R$ is isomorphic to that of $A$.  Therefore, all that remains to prove is that 
 ${\ff m}R$ is the maximal ideal of $R$.

 We  claim  that $\Omega(R) = \Omega(A)$, and we prove first that 
 $\Omega(A)$ is a full $V$-submodule of $\Omega_{F/Q}$, where $Q$ and $F$ are the quotient fields of $U$ and $V$, respectively.  
 To simplify notation let $d = d_{F/Q}$.  %We define (in notation consistent with that of Chapters 3 and 4): $$\Omega(B) = \sum_{b \in B}Sd(b).$$
Let $f \in F$.  Then $f = \frac{a}{b}$ for some $a,b \in A$, so by the quotient rule, $$d(f) = \frac{a\cdot d(b)  - b\cdot d(a)}{b^2},$$ and hence $b^2\cdot d(f) \in \Omega(A)$.  Since $\Omega_{F/Q}$ is generated as an $F$-vector space by elements of the form $d(f)$, it follows that $\Omega(A)$ is a full $V$-submodule of $\Omega_{F/Q}$.

Next we show that $R = V \cap d^{-1}(\Omega(A))$.  Let $T = V \cap d^{-1}(\Omega(A))$, so that the claim is then that $R = T$.
By Theorem~\ref{DVR correspondence},  either $T = V$ or $T$ is a bad stable domain with normalization $V$.  Either way we have by the minimality of $R$ among the class of stable rings between $A$ and $V$ that $R \subseteq T$, so we need only argue the reverse inclusion that $T \subseteq R$.  Now by Theorem~\ref{DVR correspondence}, 
 $R = V \cap d^{-1}(\Omega(R))$.  But since clearly $\Omega(A) \subseteq \Omega(R)$, we see that $$T = V \cap d^{-1}(\Omega(A)) \subseteq  V \cap d^{-1}(\Omega(R)) = R.$$
This proves that $R = T$.   In fact, since $\Omega(R)$ and $\Omega(A)$ are full $V$-submodules of $\Omega_{F/Q}$, and as we have just shown, $V \cap d^{-1}(\Omega(A)) =  V \cap d^{-1}(\Omega(R))$, then from Theorem~\ref{DVR correspondence} we conclude that $\Omega(A) = \Omega(R)$.  

  Having established that $R = V \cap d^{-1}(\Omega(A))$, we now use this fact to prove that $R$ has maximal ideal ${\ff m}R$.
Since $V$ is a DVR with maximal ideal ${\ff M}_V  = {\ff m}V$ and the maximal ideal ${\ff m} \cap U$ of $U$ extends to ${\ff m}V$ (we are using here the immediacy of the extension $U \subseteq V$), there exists $t \in {\ff m} \cap U$ such that ${\ff M}_U = tU$ and  ${\ff M}_V = tV$.
By Theorem~\ref{lifting theorem}(1)    
   there is an isomorphism of $U$-algebras: $$R/tR \rightarrow V/tV \star \Omega(A)/t\Omega(A).$$ As such, this isomorphism  carries the maximal ideal of $R/tR$ to the maximal ideal of the image, and so induces the  isomorphism:
    $$\alpha:M/tR
 \rightarrow MV/tV \star \Omega(A)/t\Omega(A): m+tR \rightarrow (m +tV, d(m) + t\Omega(A)).$$ 
 %given by $\alpha(m+tR) = (m +tV, d(m) + t\Omega(A))$ for all $m \in M$.
 Since $tV \subseteq MV \subsetneq V$ (the properness of the last inclusion is a consequence of the fact that $R \subseteq V$ is integral) and ${\ff m}V =tV$ is the maximal ideal of $V$, it must be that $MV =tV$.  Therefore, the mapping $\alpha$ induces an isomorphism of $U$-modules: $$\beta:M/tR \rightarrow \Omega(A)/t\Omega(A):m+tR \mapsto d(m) + t\Omega(A).$$ Write ${\ff m} = (x_1,\ldots,x_n)A$.  We show that $M = (t,x_1,\ldots,x_n)R$, hence proving that $M = {\ff m}R$.  In light of the isomorphism $\beta$ it suffices to prove  $$\Omega(A) = Ud(x_1) + \cdots + Ud(x_n) + t\Omega(A).$$      Moreover, since $V = U + tV$ and  $\Omega(A)$ is generated as a $V$-module by the elements $d(b)$, $b \in A$, it is enough to show that for all $b \in A$, $$d(b) \in Ud(x_1) + \cdots + Ud(x_n) + t\Omega(A).$$
 Let $b \in A$.
Since $V = A + {\ff m}V$, then $A/{\ff m} \cong V/{\ff mV}$.  But since $U \subseteq V$ is immediate, $U$ and $V$ have the same residue field, and hence $A/{\ff m} \cong V/{\ff m}V \cong U/({\ff m} \cap U).$  Therefore, $A = U + {\ff m} =   U + (x_1,\ldots,x_n)A$, and we may write $b = a + x_1b_1 + \cdots + x_nb_n$ for some $a \in U$ and $b_i \in A$.  Thus since $d$ is $U$-linear, we have:
 \begin{eqnarray*}
 d(b) &=& d(x_1b_1) + \cdots +d(x_nb_n) \\
 \: & = &  b_1d(x_1) + \cdots + b_nd(x_n) + x_1d(b_1) + \cdots + x_nd(b_n) \\
 \: & \in & b_1d(x_1) + \cdots + b_nd(x_n) + t\Omega(A),
 \end{eqnarray*}
 where the last assertion is a consequence of the fact that $x_1,\ldots,x_n \in tV$.
 Now  since $V = U + tV$, we may for each $i$ write $b_i = a_i + ts_i$ for some $a_i \in U$ and $s_i \in V$.  Therefore, \begin{eqnarray*}
 d(b)
   & \in & b_1d(x_1) + \cdots + b_nd(x_n) + t\Omega(A) \\
  \: & = &
 a_1d(x_1) + \cdots +a_nd(x_n) + t\Omega(A) \\
 \: &  \subseteq &  Ud(x_1) + \cdots + Ud(x_n) + t\Omega(A),
 \end{eqnarray*} and this verifies that $M = {\ff m}R$.
\end{proof}

In \cite{OlbGFF}, a method due to Heinzer, Rotthaus and Sally from \cite{HRS} is used in the following way to link stable rings to prime ideals in the generic formal fiber.  Let $B$ be 
  a local Noetherian domain tightly dominated by a DVR $W$ and having quotient field $Q(B)$, and let $P$ be the kernel of the canonical mapping $\widehat{B} \rightarrow \widehat{W}$.  Then a ring $S$ properly between $B$ and $W$ is a bad Noetherian stable ring tightly dominating $B$ if and only if $S = Q(B) \cap (\widehat{B}/J)$, where $J$ is a $P$-primary ideal of $\widehat{B}$ with  $P^2 \subseteq J \subsetneq P$  
%   ideal of $\widehat{B}$ such that $P^2 \subseteq J \subseteq P$ and $(J:_{\widehat{B}} b) = J$ for all $0 \ne b \in B$ 
\cite[Theorem 5.3]{OlbGFF}.  Thus there is a smallest bad stable ring tightly dominating $B$ and having normalization $W$, and it is given by $S = F \cap (\widehat{B}/P^{(2)})$, where $P^{(2)}$ is the second symbolic power of $P$. 
%
%
%when $$J = \{x \in \widehat{B}:bx \in P^2 {\mbox{ for some }} 0 \ne b \in B\}.$$ 
The   embedding dimension of this smallest stable ring $S$ is $1$ more than the embedding dimension of $\widehat{B}_P$  
   \cite[Theorem 5.3]{OlbGFF}.  
   
   Thus, returning to our context, we have from Theorem~\ref{tightly} that $R = \Ker d_{V/A}$ tightly dominates $A$, and hence this ring falls into the classification above.  But since it is by the theorem the smallest bad stable ring between $A$ and $V$ and it tightly dominates $A$, it is must be the ring defined using the ideal $P^{(2)}$. Therefore, using also Theorem~\ref{stable theorem}(2) to calculate  embedding dimension, 
       we have the following description of $R = \Ker d_{V/A}$:

\begin{cor} \label{rep}
Let $U \subseteq A \subsetneq  V$ be an extension of local Noetherian domains, where $U \subseteq V$ is  an immediate extension of DVRs  and $V$ birationally dominates $A$.  Let $F$ denote the quotient field of $V$. 
 With $P$ the kernel of the canonical mapping  $\widehat{A} \rightarrow \widehat{V}$, the ring $R = \Ker d_{V/A}$ can be represented as $$R = F \cap (\widehat{A}/P^{(2)}),$$ and the embedding dimension of $R$ is given by: 
$$\embdim R = 1 + \epsilon_R = 1+\embdim \widehat{A}_P.$$  
\end{cor}

%\begin{proof}
% By Theorem ?, $R$ tightly dominates $B$, so   by (1) of Theorem~\ref{HRS theorem}, $R = F \cap (\widehat{B}/J)$, where $J$ is an ideal of $\widehat{B}$ such that $P^2 \subseteq J \subseteq P$ and $(J:_{\widehat{B}} b) = J$ for all $0 \ne b \in B$.  But by Lemma~\ref{minimal stable}, $R$ is the unique smallest Noetherian bad stable domain between $B$ and $S$.  Thus by Theorem~\ref{HRS theorem}(2), $J = I$, where $I$ is as in (4).
% \end{proof}

We next consider the case where $A$ is essentially of finite type over $U$.  
%It is important in our arguments that $A$ occurs within an immediate extension of DVRs.   
In this case the embedding dimension of $R$ depends ultimately only on the field extension $F/Q$, where $Q$ and $F$ are the quotient fields of $U$ and $V$, respectively.

\begin{thm} \label{eft} Let $U \subseteq A \subsetneq  V$ be an extension of local Noetherian domains, where $U \subseteq V$ is  an immediate extension of DVRs with quotient fields $Q$ and $F$, respectively,  and $V$ birationally dominates $A$.
Let $R = \Ker d_{V/A}$.  If $A$ is essentially of finite type over $U$ with Krull dimension $d>1$, then $$\embdim R = 1 + \dim_F \Omega_{F/Q},$$ and 
 for  $P$ the kernel of $\widehat{A} \rightarrow \widehat{V}$, $$\embdim \widehat{A}_P = \dim_F \Omega_{F/Q}.$$
\end{thm}

\begin{proof}
Since $A$ is the localization of a finitely generated $U$-subalgebra of $V$ having quotient field $F$, the field $F$ is a finitely generated field extension of $Q$. Moreover,
 since $A$ has Krull dimension $>1$, $F$ is not algebraic over $Q$, and thus  $\Omega_{F/Q} \ne 0$ (2.2).
 %since $F$ is not separably algebraic over $Q$
% \cite[Corollary 16.16, p.~402]{Eis}.
Also, since the image under $
d_{F/Q}$ of any generating set of the field extension $F/Q$ generates the $F$-vector space  $\Omega_{F/Q}$, then $\Omega_{F/Q}$ has finite dimension as an $F$-vector space.  Write
 $A = U[x_1,\ldots,x_n]_{\ff p}$ for some $x_1,\ldots,x_n \in A$ and prime ideal ${\ff p}$ of $U[x_1,\ldots,x_n]$.
Let  $K$ be a
 free $V$-submodule of $\Omega_{F/Q}$ such that $\rank(K) = \dim_{F} \Omega_{F/Q}$ and $d_{F/Q}(x_1),\ldots, d_{F/Q}(x_n) \in K$.
 %(since $V$ is a DVR, finitely generated $V$-submodules of $\Omega_{F/Q}$
%are free, and so such a choice of $K$ is possible).
Set $T =
V \cap d_{F/Q}^{-1}(K)$, and observe that since $d_{F/Q}$ is $U$-linear
and $d_{F/Q}(x_i) \in K$ for all $i$, it follows that $U[x_1,\ldots,x_n] \subseteq T$.  Since $V$ birationally dominates $A$, it must be that ${\ff p} = {\ff M}_V \cap U[x_1,\ldots,x_n]$.  Now
 since $V$ is the normalization of $T$ (Theorem~\ref{DVR correspondence}), 
  $T$ is a quasilocal ring with maximal ideal ${\ff M}_V \cap T$.  Thus the maximal ideal of $T$ contracts in  $U[x_1,\ldots,x_n] $ to ${\ff p}$, and we conclude that $A \subseteq T$.
 By Theorem~\ref{DVR correspondence}, $T \subsetneq V$, $K = \Omega(T)$ and $T$ is a bad stable ring with normalization $V$.  
By Theorem~\ref{stable ed}(2),
 $$\embdim T = 1 + \dim_{V/{\ff M}_V} K/{\ff M}_VK = 1+\rank(K) = 1+\dim_F \Omega_{F/Q}.$$

Next we consider the ring $R = \Ker d_{V/A}$.  By Theorem~\ref{tightly}, $R$ is a bad stable domain contained in every bad stable ring between $A$ and $V$. In particular,
 $R \subseteq T$. We claim that $\embdim R = \embdim T$.  By Theorem~\ref{stable ed}, 
 $\embdim R = 1 + \epsilon_R \leq 1+ \dim_F \Omega_{F/Q}.$   Moreover, since $T$ is between the stable ring $R$ and the normalization of $R$, $\embdim T \leq \embdim R$ \cite[Theorem 4.4]{OlbGFF}.  Therefore,  
 $$1+\dim_F \Omega_{F/Q} = \embdim T \leq  \embdim R \leq 1+ \dim_F \Omega_{F/Q},$$
and hence  $\embdim R = 1 + \dim_F \Omega_{F/Q}$.  Combining this with Corollary~\ref{rep} completes the proof of the theorem.
\end{proof}

\begin{cor} With the same assumptions as the theorem, if also  $F$ is separable over $Q$, then $\embdim R = d$ and  the ring $\widehat{A}_P$ is a regular local ring.  
\end{cor}

\begin{proof}
 Since  $F$ is separable over $Q$, and $F$ is a finitely generated extension of $Q$, the dimension of $\Omega_{F/Q}$ is the transcendence degree of $F$ over $Q$ \cite[Corollary 16.17(a), p.~403]{Eis}.
%see Proposition~\ref{differential basis}(5).  
Thus by Theorem~\ref{eft}, $\embdim R = 1 + $tr.deg$_QF$.
 Since $U$ is DVR, it is universally catenary (see for example \cite[Corollary 18.10, p.~457]{Eis}), and hence, since $A$ is essentially of finite type over $U$, the dimension formula holds for the extension $U \subseteq A$ \cite[Theorem 15.5, p.~118]{Ma}.  In particular, $$\dim(A) + {\mbox{tr.deg}}_{U/({\ff m} \cap U)} A/{\ff m} = \dim(U) + {\mbox{tr.deg}}_Q F.$$  As noted in the proof of Theorem~\ref{tightly}, $A = U + {\ff m}$, so it follows that $A/{\ff m} \cong U/({\ff m} \cap U)$.  Thus since $\dim(U) = 1$, the dimension formula yields $d= \dim(A) = 1 + {\mbox{tr.deg}}_Q F$, so that $$\embdim R = 1 +  {\mbox{tr.deg}}_Q F = 1 + (d-1) = d.$$ The corollary now follows from Theorem~\ref{eft}. 
  \end{proof}
  
  Thus if the quotient field $Q$ of $U$ is perfect and $A$ is essentially of finite type over $U$, then $\widehat{A}_P$ is a regular local ring for every dimension $1$ prime ideal $P$ of $\widehat{A}$ with $P \cap A = 0$.  This however is just a special case of the fact that a DVR with perfect quotient field is a $G$-ring; see  \cite[Remark 10.1]{HOST}.

  \section{Analytic ramification in positive characteristic}
  
  In this section we show that all bad stable domains in positive characteristic arise  as in Section 4 from derivations.  The reason for this is that 
  in positive characteristic, all one-dimensional analytically ramified local Noetherian stable domains tightly dominated by a DVR occur within an immediate extension of DVRs.    This is a consequence of a theorem of Bennett, which states that if $R$ is a one-dimensional local Noetherian domain of characteristic $p>0$, and there is a nilpotent prime ideal $P$ of $\widehat{R}$ such that $\widehat{R}/P$ is a DVR, then there is a DVR  $U$ such that $U \subseteq R \subseteq \widehat{U}$ and $R^{q} \subseteq U$ for some $q=p^e$  \cite[Theorem 1, p.~133]{Bennett}.  This implies: 

\begin{lem} \label{Bennett lemma} {\em (Bennett \cite{Bennett})} 
 If the normalization of a Noetherian domain $A$ of positive characteristic is a DVR $V$  that  tightly dominates $A$, then $A$ contains a DVR $U$ such that $U \subseteq V$ is an immediate extension.  
\end{lem}

We apply the lemma in Theorem~\ref{p theorem} to show that all bad Noetherian stable domains in positive characteristic occur as the pullback of a derivation. The theorem depends on two lemmas that are valid in any characteristic.

  \begin{lem} \label{easy lead in} Let $V$ be a DVR with quotient field $F$, and let $Q$ be a subfield of $F$.  Then $V$ is an immediate extension of the DVR $V \cap Q$ if and only if $F = V + Q$. 
\end{lem} 

\begin{proof} Let $U = V \cap Q$.  If $V$ is an immediate extension of $U$, then $V/U$ is a torsion-free divisible $U$-module.   Moreover, since $V$ is a DVR dominating $U$, it follows that $F = QV$.  Thus from the facts that $F = QV$ and $V/U$ is divisible, we deduce  that  $F = QV = V + Q$.   Conversely, if $F = V + Q$, 
then $V/U = V/(Q \cap V) \cong (V+Q)/Q = F/Q$, so that $V/U$ is a divisible torsion-free $U$-module, and hence  $U \subseteq V$ is an immediate extension of DVRs.
\end{proof}

  \begin{lem}  \label{FRGL}  Let $V$ be a DVR with quotient field $F$, let $K$ be a torsion-free $V$-module such that  $K/{\ff M}_VK$ is a nonzero finitely generated   $V$-module, and suppose   $D:V \rightarrow FK$ is a derivation such that  $FK$ is generated as an $F$-vector space by $D(V)$. If  $V = \Ker D + cV$ for some $0 \ne c \in {\ff M}_V \cap \Ker D$,  then  $R = D^{-1}(K)$ is bad  
   Noetherian stable domain. 
  \end{lem} 
  
  \begin{proof}  Let $U = \Ker D$, and let $Q$ be the quotient field of $U$. We show first that $U \subseteq V$ is an immediate extension of DVRs.  It suffices by Lemma~\ref{easy lead in} to verify that $U = V \cap Q$ and $F = V + Q$.  
      Clearly, $U \subseteq V\cap Q$.  Let $v \in V \cap Q$, and write $v = \frac{a}{b}$ where $a,b \in \Ker D$ with $b \ne 0$.  Then since $D(b) = 0$, we have  $bD(v) = D(a) = 0$, and since $K$ is torsion-free, $D(v) = 0$.  Thus $v \in \Ker D$, which shows that $V \cap Q = U$.  Now to see that $F = V + Q$, 
     observe  that since  $U$ contains a nonzero element in the maximal ideal of $V$,  it suffices to show that $QV \subseteq V + Q$.  Let $0 \ne a \in \Ker D$ and $v \in V$. We claim $a^{-1}v \in Q + V$.    Since $V$ is  a DVR, there exists $k>0$ such that $c^k \in aV$,  and hence from the assumption that $V = \Ker D + cV$, it follows that $V = \Ker D + aV$. Thus $v = b + aw$ for some $b \in \Ker D$ and $w \in V$, and hence   $a^{-1}v = a^{-1}b + w \in Q + V$, which proves the claim that $U \subseteq V$ is immediate. 
     
     Now let $C = U \setminus \{0\}$.  As noted above, $V_C = F$.   
    % Extend $D$ to a derivation $D:F \rightarrow FK$.  
    Since by assumption  $FK$ is generated as an $F$-vector space by $D(F)$,     Proposition~\ref{pre-construction} implies that the mapping $R \rightarrow V \star K:r \mapsto (r,D(r))$ is an analytic isomorphism along $C$.  As in the proof of Theorem~\ref{stable theorem}(1), this implies that the ring $\widehat{R}$ is isomorphic to $\widehat{V} \star \widehat{K}$.  In particular, since $K/{\ff M}_VK$ is a nonzero $V$-module, then $\widehat{K} \ne 0$ and hence  there is a nonzero prime ideal $P$ of $\widehat{R}$ such that $P^2 = 0$ and $\widehat{R}/P$ is a DVR.  Thus by (2.5), $R$ is a bad  stable domain.   Also, as in the proof of Theorem~\ref{stable theorem}(2), since $K/{\ff M}_VK$ is a finitely generated $V$-module, $R$ is a Noetherian domain.   
 \end{proof}
  
  \begin{thm} \label{p theorem}  The following are equivalent for a  quasilocal domain $(R,{\ff m})$ of positive characteristic.  
  
  \begin{itemize}
  
  \item[(1)]  $R$ is a bad Noetherian stable domain.
  
  \item[(2)]  The normalization $V$ of $R$  is a DVR and there
  exists a torsion-free $V$-module $K$ with $K/{{\ff M}_V}K$ a nonzero finitely generated $V$-module,
  and a derivation $D:V \rightarrow FK$  such that $D(V)$ generates $FK$ as an $F$-vector space, $R = D^{-1}(K)$ and $V = \Ker D + {\ff m}V$.   
  \end{itemize} 
  \end{thm} 
  
\begin{proof}
Suppose that $R$ is a bad Noetherian stable domain.  Then by (2.4), $V$ is a DVR that tightly dominates $R$, and hence by Lemma~\ref{Bennett lemma} there exists a DVR $U \subseteq R$ such that $U \subseteq V$ is an immediate extension.  By Theorem~\ref{DVR correspondence}, $R = V \cap d^{-1}_{F/Q}(\Omega(R))$, where $Q$ and $F$ are the quotient fields of $U$ and $V$, respectively.   By Theorem~\ref{stable theorem}(2), $\Omega(R)/{\ff M}_V\Omega(R)$ is a nonzero finitely generated $V$-module.   Let $D$ be the restriction of $d_{F/Q}$ to $V$. 
By definition, $D(R)$ generates $\Omega(R)$ as a $V$-module, and hence $D(V)$ generates $F\Omega(R) = \Omega_{F/Q}$ as an $F$-vector space.   
   Moreover, $R = D^{-1}(\Omega(R))$, and  since $D$ is $U$-linear and $V/U$ is a divisible $U$-module, it follows that $V = \Ker D + cV$ for all $0 \ne c \in U$. Hence $V = \Ker D + {\ff m}V$, and   this verifies statement (2).  The converse, that (2) implies (1), 
in given by Lemma~\ref{FRGL}.           
\end{proof}

These ideas also lead to a characterization in Theorem~\ref{Bennett case} of the DVRs of positive characteristic which are the normalization of an analytically ramified local Noetherian domain. 

%The theorem relies on  
%a  lemma that concerns the situation that was the focus of the last two sections, the case in which a DVR $V$ is an immediate extension of a DVR subring $U$.  In this case, 

\begin{thm}   \label{Bennett case} The following are equivalent for a DVR $V$ having quotient field $F$  of  positive characteristic $p$.  
\begin{itemize}

\item[(1)]  $V$ is the normalization of an analytically ramified local Noetherian domain.

\item[(2)]  There exists a subfield $Q$ of $F$ such that $F = V + Q$ and $F \ne Q[F^p]$. 

%\item[(3)]  There exists a DVR $U \subseteq V$ such that this extension is immediate and  $F \ne Q[F^p]$, where $Q$ is the quotient field of $U$.  

\end{itemize}  
 
\end{thm}  

\begin{proof}  To see that (1) implies (2), let $R$ be an analytically ramified local Noetherian domain having normalization $V$.   As discussed before Corollary~\ref{main theorem 1}, every one-dimensional analytically ramified local Noetherian domain has a bad $2$-generator  overring.  Thus there exists a bad Noetherian stable ring $S$ having normalization $V$, and by (2.4), $S$ is tightly dominated by $V$.  Thus by Lemma~\ref{Bennett lemma}, there exists a DVR $U \subseteq S$ such that $U \subseteq V$ is an immediate extension.  Let $Q$ be the quotient field of $U$.  Then by Corollary~\ref{main theorem 1}, $F \ne Q[F^p]$, and since $U \subseteq V$ is immediate, $V/U$ is a divisible $U$-module and $QV = F$.  Thus if $0 \ne u \in U$ and $v \in V$, writing $v = u' + uw$ for some $u' \in U$ and $w \in V$, we have  $u^{-1}v = u^{-1}u' + w \in Q + V$, which, since $F= QV$, proves that  
 $F  =  V +Q$.  The converse, that  (2) implies (1), follows from    Corollary~\ref{main theorem 1} and Lemma~\ref{easy lead in}. \end{proof}

\section{Prescribed normalization}

The DVRs of positive characteristic that are the normalization of an analytically ramified local Noetherian domain were characterized in Theorem~\ref{Bennett case}. In this section we describe additional cases where a given DVR is the normalization of an analytically ramified local ring. The first case involves equicharacteristic DVRs with ``large'' quotient field.

 % Our main tool in this regard is Corollary~\ref{main theorem 1} of the last section.   

%\begin{lem} {\em \cite[Lemmas 3.1 and~3.2]{OlbFR}} \label{dig} Let $ F/k $ be a separably generated field extension  of infinite transcendence degree, where 
% $k$ has at most countably many elements.  Suppose that  $S$ is a
%$k$-subalgebra of $F$ having quotient field $F$.  Then  
% for any $t \in S$ there exists a ring $A$ such that $k[t] \subseteq  A \subseteq S$ and
%  $S/A$ is a torsion-free divisible
%$A$-module and, with $Q$ the quotient field of $A$, $\dim_F \Omega_{F/Q}$ is infinite.
%\end{lem}

\begin{thm}   Let $V$ be an equicharacteristic DVR with quotient field $F$ having a countable subfield $k$ contained in $V$ such that  
 $F$ is separably generated and of infinite transcendence degree over  $k$.  Then for every $d>1$, $V$ is the normalization of a bad Noetherian stable ring of embedding dimension $d$.  Moreover, $V$ is the normalization of a bad stable domain that is not Noetherian.
 \end{thm}
 
 \begin{proof}   In \cite[Lemmas 3.1 and~3.2]{OlbFR} it is shown that since $ F/k $ is a separably generated field extension  of infinite transcendence degree, where 
 $k$ has at most countably many elements and $V$ is a
$k$-subalgebra of $F$ having quotient field $F$, then  
 for any $t \in V$ there exists a ring $U$ such that $k[t] \subseteq  U \subseteq V$,
  $V/U$ is a torsion-free divisible
$U$-module and, with $Q$ the quotient field of $U$, $\dim_F \Omega_{F/Q}$ is infinite.
Thus, choosing  $t \in V$ such that $tV$ is the maximal ideal of $V$, there exists a subring $U$ of $V$ such that $V/U$ is a torsion-free divisible module and $t \in U$. 
The fact that $V/U$ is torsion-free implies that $U = Q \cap V$.  Also since $V/U$ is divisible and $F = VQ$, then  $F = V + Q$, so that by Lemma~\ref{easy lead in},  $U \subseteq V$ is an immediate extension of DVRs with $\dim_{F} \Omega_{F/Q}$ infinite.  Theorems~\ref{DVR correspondence} and~\ref{stable theorem} now complete the proof.
\end{proof}

Since fields extensions in characteristic $0$ are separable, we deduce  
 
\begin{cor} \label{dig cor} If $V$ is a DVR of equicharacteristic $0$ that has infinite transcendence degree over its prime subfield, then for each $d>1$, $V$ is the normalization of a bad Noetherian stable domain of embedding dimension $d$. \qed
\end{cor}

 In particular, every uncountable DVR of equicharacteristic $0$ is the normalization of an bad Noetherian stable domain.  
 
 Next we consider the case of algebraic function fields $F/k$, and  we characterize the DVRs in $F/k$ that are the normalization of an analytically ramified local Noetherian $k$-algebra.  The emphasis in the theorem is thus on staying inside the category of $k$-algebras.  By a  DVR {\it in $F/k$} we mean a DVR  that is a $k$-algebra having quotient field $F$.  A {\it divisorial valuation ring} in $F/k$ is a DVR $V$ in $F/k$ such that $\trdeg_k V/{\ff M}_V = \trdeg_k F -1.$

 \begin{thm} \label{function field}
 Let $F/k$ be a finitely generated field extension, and let $V$ be a DVR in $F/k$ such that  
$V/{\ff M}_V$ is a finitely generated extension of $k$.  Then the following statements are equivalent. 
\begin{itemize}

\item[(1)]  
 $V$ is the normalization of an analytically ramified local Noetherian domain  containing $k$.

\item[(2)]   $V$ is the normalization of a bad stable ring   containing $k$ and having embedding dimension $d = \trdeg_k F - \trdeg_k V/{\ff M}_V$.

 \item[(3)]   $V$ is not a divisorial valuation ring in $F/k$.

\end{itemize}  
   
\end{thm}  

\begin{proof} That (2) implies (1) is clear.     To see that (1) implies (3), let $R$ be an analytically ramified local ring having normalization $V$, and suppose by way of contradiction that $V$ is a divisorial valuation.  Then, as noted in the proof of Theorem~\ref{Bennett case}, we may assume without loss of generality that $R$ is a bad stable ring having normalization $V$; as such, since by (2.4), $V/R$ is a divisible $R$-module, $R$ has the same residue field as $V$.  
Thus with $n = \trdeg_{k} F$ and $M$ the maximal ideal of $R$, there exist units $x_1,\ldots,x_{n-1} \in R$ such that the images in $R/M$ of these elements are algebraically independent over $k$.  Therefore, if $0 \ne f \in k[X_1,\ldots,X_{n-1}]$, then  $f(x_1,\ldots,x_{n-1}) \not \in M$, and this implies that   
 $k(x_1,\ldots,x_{n-1}) \subseteq R$.  Since $R$ has quotient field $F$, 
 %and $\trdeg_k F  =n$, 
 there is a finitely generated $k(X_1,\ldots,X_{n-1})$-subalgebra $A$ of $R$ with quotient field $F$.  Since $\trdeg_k F = n$, it follows that $A$ is a one-dimensional Noetherian domain, and as a finitely generated algebra over a field, $A$, and hence every overring of $A$, has finite normalization. (Indeed, if $B$ is an overring of $A$ and $\overline{A}$ is the normalization of $A$, then since $\overline{A}$ is a Dedekind domain, it follows that $B\overline{A}$ is the normalization of $B$. Thus since $A$ has finite normalization, then so does $B$.)
     But since $R$ is analytically ramified, $R$ does not finite normalization, and this contradiction implies that $V$ is not a divisorial valuation ring.  
 
% Since $R$ has quotient field $F$ and $F$ is a finitely generated extension of $k$, there exists a finitely generated $k$-subalgebra $A$ of $R$ having quotient field $F$.  Moreover, we may assume $A$ contains $x_1,\ldots,x_{n-1}$.  Let ${\ff p} = M \cap A$.      By the Dimension Formula, \begin{center}ht $M + \trdeg_{A_{\ff p}/{\ff p}A_{\ff p}}R/M = $ ht ${\ff p}.$\end{center}  Since $x_1,\ldots,x_{n-1} \in A$, it follows that $R/M$ is algebraic over $A_{\ff p}/{\ff p}A_{\ff p}$.  Hence $1 = $ ht $M = $ ht ${\ff p}$, so that $A_{\ff p}$ is a one-dimensional Noetherian ring essentially of finite type over a field.  As such, it has a finite normalization, as does every overring, including $R$.  This contradicts the assumption that $R$ is a bad stable ring, so we conclude that $V$ is not a divisorial valuation ring. 
Finally, to see that (3) implies (2), let $r = \trdeg_{k} V/{\ff M}_V$.  Then as above, there exist algebraically $k$-independent elements $x_1,\ldots,x_r$ of $V$ such that $k':=k(x_1,\ldots,x_r) \subseteq V$.  Let $A$ be a finitely generated $k'$-subalgebra of $V$ with quotient field $F$.  
Then since $V/{\ff M}_V$ is algebraic over  $k'$, ${\ff m} := {\ff M}_V \cap A$ is a maximal ideal of $A$.  Since also $V/{\ff M}_V$ is finite over $k'$, we may enlarge $A$ by adjoining finitely many elements of $V$ to obtain that $V$ tightly dominates $B:=A_{\ff m}$.   Now $B$ has Krull dimension $ \trdeg_k F - r = d>1$ (here we are using the assumption that $r < d-1$),  so as discussed before Corollary~\ref{rep}, there is a bad stable ring $S$  having normalization $V$ and embedding dimension $1$ more than the embedding dimension of $\widehat{B}_P$, where $P$ is a certain dimension one prime ideal of $\widehat{B}$ with $P \cap B = 0$.  Since $B$ is excellent, the ring $\widehat{B}_P$ is a regular local ring,  and hence its embedding dimension is $d-1$.  Therefore, $S$ has embedding dimension $d$.       
\end{proof}

We next consider  two  cases, one in which   $V$ is a DVR that contains a coefficient field and the other a special instance of mixed characteristic.  In the former case we can describe some of the bad stable subrings of $V$, while in the latter we can describe {all} the bad stable subrings of $V$ which share their quotient field with $V$ and for which the relevant prime integer $p$ of $V$ is not a unit.

\begin{prop} \label{coefficient field case}  \label{characteristic p case} Let $V$ be a DVR with residue field $k$ and quotient field $F$, and  let $t \in V$ be such that $tV$ is the maximal ideal of $V$.  Then the following statements hold for $V$.
\begin{itemize}
\item[(1)]  Suppose  $V=k+tV$.  If $\Omega_{F/k(t)} = 0$, then no ring containing $k[t]$ and  having  normalization $V$ is an {analytically ramified local Noetherian domain}.   Otherwise, if $\Omega_{F/k(t)} \ne 0$, then  there is a one-to-one correspondence between the proper full  $V$-sub\-mod\-ules of $\Omega_{F/k(t)}$ and the
bad stable rings   containing
 $k[t]_{(t)}$ and having normalization $V$.

\item[(2)]  Suppose $V$ has characteristic $0$, $k$ has exactly $p$ elements ($p$ a prime), and $pV$ is the maximal ideal of $V$.  If ${F/{\mathbb{Q}}}$ is algebraic, where ${\mathbb{Q}}$ is the field of rational numbers, then no proper subring of $V$ having quotient field $F$ and normalization $V$ is an {analytically ramified local Noetherian domain}.  Otherwise, if $F/{\mathbb{Q}}$ is not algebraic, then
 there is a one-to-one correspondence between the full  $V$-submodules of $\Omega_{F/{\mathbb{Q}}}$ and the
quasilocal one-dimensional stable rings $R$ birationally dominated by $V$ such that $pR \ne R$.

\index{stable domain!existence of} \index{stable domain!Noetherian}
\item[(3)]  For any such proper stable subring $R$ of $V$ arising as in (1) or (2),
  $R$ is a bad stable domain with normalization $V$.  Moreover, $R$ can be chosen to be Noetherian of embedding dimension $n+1$, as long as $n$ is selected in (1) such that $1\leq n  \leq \dim_{F} \Omega_{F/k(t)}$, and in (2) such that  $1\leq n \leq \dim_F \Omega_{F/{\mathbb{Q}}}$.  If either of these vector space dimensions is infinite, then in the corresponding case, $R$ can be chosen to be a non-Noetherian bad stable domain.

\end{itemize}
\end{prop}

%\begin{prop} \label{coefficient field case} Let $V$ be a DVR with quotient field $F$, and suppose that  $V$ contains a field $k$ such that the residue field of $V$ is $k$. Let $t \in V$ be such that $tV$ is the maximal ideal of $V$.
%Then there is a one-to-one correspondence between the $k[t]$-subalgebras $R$ of $V$ having quotient field $F$ and the full $V$-submodules of $\Omega_{F/k(t)}$.
%Moreover, for any such proper $k$-subalgebra $R$ of $V$,
 % $R$ is a bad stable domain, $V$ is the integral closure of $R$ and  $R$ is Noetherian with embedding dimension $\epsilon_R+1$ if and only if $\epsilon_R$ is finite.

%\end{prop}

\begin{proof}
(1)
With $U:=k[t]_{(t)}$, the extension $U \subseteq V$ is immediate,  so by
Theorem~\ref{DVR correspondence} there is a one-to-one correspondence between the nonzero full $V$-submodules of $\Omega_{F/k(t)}$ and the  stable rings $R$  occurring between $U$ and $V$ and having normalization $V$.  Suppose that $R$ is an  analytically ramified local Noetherian domain having normalization $V$ and containing $U$.  Then  $R$ has dimension $1$, and by Theorem~\ref{tightly} there exists a bad Noetherian stable ring between $R$ and $V$, in which case by the correspondence, $\Omega_{F/k(t)} \ne 0$.    

% Moreover, if $R$ is any one-dimensional stable ring such that $U \subseteq R \subsetneq V$ and $V$ birationally dominates $R$, then since $V/U$ is a divisible $U$-module, it follows that $V/R$ is a divisible $R$-module (Proposition~\ref{Matlis}), and hence that $R$ is bad stable domain.

(2)
We view
 $U:={\mathbb{Z}}_{(p)}$ (the ring of integers localized at the prime ideal $(p)$) as a subring of $V$.    Then $U = V \cap {\mathbb{Q}}$, and since the residue field of $V$ has $p$ elements and the maximal ideal of $V$ is $pV$, we have $V = U+pV$.  Hence $U \subseteq V$ is an immediate extension.  %It follows that $V/U$ is a torsion-free divisible $U$-module, and since every nonzero ideal of $V$ contains a power of $p$, the extension $U \subseteq V$ has TGF. Therefore, by Proposition~\ref{main char basic}, $U \subseteq V$ is an analytic extension.
 Statement (2) follows now as in (1) from
 Theorems~\ref{DVR correspondence}  and~\ref{tightly}. (We are using implicitly here that $F/{\mathbb{Q}}$ is algebraic if and only if $\Omega_{F/{\mathbb{Q}}} = 0$; see (2.2).)

(3)  Let $Q$ denote the quotient field of $U$, where $U$ is as in (1) or (2), depending on which case we wish to consider.  Let $n$ be such that $n \leq \dim_{F} \Omega_{F/Q}$, where we admit the possibility that $n$ is infinite.  Let $K$ be a full $V$-submodule of $\Omega_{F/Q}$ such that $K$ is the direct sum of a rank $n$ free $V$-submodule of $\Omega_{F/Q}$ and a divisible $V$-submodule.  
 With $R:=
 V \cap d_{F/Q}^{-1}(K)$, we have by Theorem~\ref{DVR correspondence}
  that   $R$ is a bad stable domain.  Moreover, by Theorem~\ref{stable ed}, $R$ 
   is Noetherian of embedding dimension $n+1$ if and only if $n$ is finite.  Also, by Proposition~\ref{stable theorem},  $n \leq \dim_F \Omega_{F/Q}$.
 This proves (3).
\end{proof}

In Theorem~\ref{catalog complete DVR} we apply the proposition to determine when certain complete DVRs are the normalization of an analytically ramified local Noetherian domain.  The theorem depends on the next lemma, where
 we calculate the dimension of $\Omega_{k((X))/k(X)}$ in two cases.  In the first case in which $k$ is perfect of characteristic $p \ne 0$ (i.e., $k = k^p$), we see that the module of K\"ahler differentials is trivial, while in a case at the other extreme, where not only is $k \ne k^p$ but $[k:k^p]$ is uncountable, then this module has infinite dimension as a $k((X))$-vector space.
Statement (2) of the lemma is mentioned without proof in \cite[Remark 3.7(i)]{FR}.  We supply a proof here for the sake of completeness.

\begin{lem}  \label{perfect power series}
Let $
k$ be a  field of positive characteristic $p$, and  let $X$
be an indeterminate for $k$.

\begin{itemize}

\item[(1)]  If $k$ is perfect,  then $\Omega_{k((X))/k(X)} = 0$.

\item[(2)]  If $[k:k^p]$ is uncountable, then $\Omega_{k((X))/k(X)}$ has infinite dimension as a $k((X))$-vector space.

%\item[(3)]  If $k$ has characteristic $0$, then $\Omega_{k((X))/k(X)}$
%has infinite dimension as a $k((X))$-vector space.
\end{itemize}
\end{lem}

\begin{proof} (1)  Let $F = k((X))$.  Then to show that $\Omega_{F/k(X)} = 0$,  it is by  (2.2) sufficient (and necessary) to show that $F = k(X)[F^p]$.
Moreover,   since $k(X)[F^p]$ is a field and $F$ is the quotient field of $k[[X]]$, we need only show that $k[[X]] \subseteq k(X)[F^p]$.  To this end, let $z \in k[[X]]$, and write $z = \sum_{i=0}^\infty \alpha_i X^i$, where $\alpha_i \in k$.  For each $j = 0,1,\ldots,p-1$, define $$z_j = \sum_{i=0}^\infty \alpha_{ip+j} X^{ip+j},$$ so that $z = z_0 + z_1 + \cdots + z_{p-1}.$  Then since
$k$ is perfect of characteristic $p$,
 $$z_j =  \sum_{i=0}^\infty \alpha_{ip+j} X^{ip+j}=X^j \sum_{i=0}^\infty \alpha_{ip+j} X^{ip} = X^j \left(\sum_{i=1}^\infty (\alpha_{ip+j})^{1/p}X^i\right)^p.$$
Therefore, $z_j \in k(X)[F^p]$, so that $z = z_0 + \cdots + z_{p-1} \in k(X)[F^p]$, proving that $F = k(X)[F^p]$.

(2)
Let $F = k[k^p((X^p))]$.
We first claim that
 $[k((X)):F]$ is infinite.  Suppose the contrary.  Then there exist $z_1,\ldots,z_n \in k((X))$ such that $$k((X)) = F z_1 + \cdots + Fz_n.$$ Now there exists $m \in {\mathbb{Z}}$ such that each $z_j$ is of the form $z_j=\sum_{i\geq m}\delta_i X^i$, where  $\delta_i \in k.$  Let $L$ be the subfield of $k$ generated by $k^p$ and all the coefficients $\delta_i$ of all the $z_j$.
Then $[L:k^p]$ is countable, so since $[k:k^p]$ is uncountable, the $L$-vector space $k$ has infinite dimension.
With this in mind, let $\alpha_0,\alpha_1,\alpha_2, \ldots $ be elements of $k$ that are linearly independent over $L$, and define $w = \sum_{i\geq 0} \alpha_i X^i.$
Then since $w \in k((X)) = Fz_1 + \cdots + Fz_n$, there exist $f_1,\ldots,f_n \in F$ such that $w = f_1z_1 + \cdots + f_nz_n$.  Now there exists $t>0$ such that for each $j=1,2,\ldots,n$, we may write $$f_j = \beta_{j1} \sum_{i} \gamma_{j1i}^p X^{ip} + \cdots + \beta_{jt} \sum_i \gamma_{jti}^p X^{ip},$$ where $\beta_{j\ell},\gamma_{jmi} \in k$, and $i$ is allowed to range over some set of integers containing a lower bound.  Then $$f_j = \sum_{i} (\beta_{j1} \gamma_{j1i}^p  + \cdots + \beta_{jt} \gamma_{jti}^p) X^{ip},$$
 and so the coefficients of $f_jz_j$ are in the $L$-vector space $L\beta_{j1} + \cdots + L\beta_{jt}$.  Hence the coefficients of $w = f_1z_1 + \cdots + f_nz_n$ are in the finite dimensional $L$-subspace of $k$ generated by $\{\beta_{j\ell}:1 \leq j \leq n,1\leq \ell \leq t\}.$  But then  $\alpha_0,\alpha_1,\ldots$, are  in this  finite dimensional  $L$-subspace of $F$, which contradicts the choice of the $\alpha_i$ as elements of $F$ which are linearly independent over $L$.  This proves that $[k((X)):F]$ is infinite.

Now to show that $\Omega_{k((X))/k(X)}$ has infinite dimension as a $k((X))$-vector space, it suffices to show that $k((X))$ has an infinite $p$-basis over $k(X)$; see \cite[Theorem 26.5, p.~202]{Ma}.   
  Thus, since $k((X))^p = k^p((X^p))$, it is enough to prove that $$[k((X)):k(X)[F]] = \infty.$$  Observe that
 $$[k((X)):F] = \left[k((X)):k(X)[F]\right] \cdot \left[k(X)[F]: F\right],$$  and since $[k((X)):F]$ is infinite and $[k(X)[F]:F]$ is finite (note that $X$ is algebraic over $F$), we conclude that $[k((X)):k(X)[F]]$ is  infinite, and this proves (2).
%
% (3)  It suffices since $k$ has characteristic $0$ to argue that the field extension $k((X))/k(X)$ has infinite transcendence degree; see \cite[Corollary A1.5(a), p.~567]{Eis}. See~\cite[p.~220]{ZS} for such an argument.
\end{proof}

\begin{thm}  \label{catalog complete DVR}  Let $V$ be a complete DVR with residue field $k$.

\begin{itemize}
\item[(1)]  If $V$ has characteristic $p \ne 0$ and $k$ is perfect, then there does not exist an analytically ramified local Noetherian domain containing $k$   whose normalization is  $V$.

\item[(2)]  If either (a) $V = \widehat{{\mathbb{Z}}}_p$, (b)  $V$ and $k$ have characteristic $0$, or (c) $V$ has characteristic $p \ne 0$ and $[k:k^p]$ is uncountable, then for every $d>1$ there exists a bad Noetherian stable domain of embedding dimension $d$ whose normalization is $V$.  There also exists a non-Noetherian bad stable domain whose normalization is $V$.

\end{itemize}
\end{thm}

\begin{proof}
(1) By the Cohen structure theorem for complete local rings, we may assume that $k \subseteq V$, where $k$ is as in (1); see \cite{Cohen}.  Suppose there exists an analytically ramified local Noetherian domain that contains $k$ and whose normalization is $V$.  Then, as noted before Corollary~\ref{Bennett case},  
a theorem of Matlis shows there exists a  bad Noetherian stable domain $R$ that contains $k$ and  whose normalization is $V$.  Since $V/R$ is a divisible $R$-module, the maximal ideal of $V$ is extended from the maximal ideal of $R$, and hence since $V$ is a DVR, there exists $t \in R$ such that $tV$ is the maximal ideal of $V$.  Then $k[t]_{(t)} \subseteq R$, and by the structure theorem, $V = k[[t]]$.  But by Lemma~\ref{perfect power series}, $\Omega_{k((t))/k(t)} = 0$, so  by Theorem~\ref{coefficient field case}(1), no such ring $R$ can exist.

(2)  If $V = \widehat{{\mathbb{Z}}}_p$, then since $\widehat{{\mathbb{Q}}}_{p}$ has infinite transcendence degree over the characteristic $0$ field ${\mathbb{Q}}$, the $\widehat{{\mathbb{Q}}}_{p}$-vector space 
 $ \Omega_{ \widehat{{\mathbb{Q}}}_p/{\mathbb{Q}} }$ has infinite dimension \cite[Corollary A1.5(a), p.~567]{Eis}, and hence the claim follows from Theorem~\ref{coefficient field case}(2) and (3).  The case where $V$ and $k$ have characteristic $0$ is decided by Corollary~\ref{dig cor}.   
 %
 %  $V = k[[t]]$ for some $t \in V$, and since $k$ has characteristic $0$, we have by Lemma~\ref{perfect power series} that $\Omega_{k((t))/k(t)}$ has infinite dimension as a vector space over $k((t))$.  Thus in this case also we may apply Theorem~\ref{coefficient field case}(2) and (3).  
 Finally, if $V$ has characteristic $p \ne 0$ (so that also $k$ has characteristic $p$) and $[k:k^p]$ is uncountable, then $V= k[[t]]$ for some $t \in V$, and by Lemma~\ref{perfect power series}(2), $\Omega_{k((t))/k(t)}$ has infinite dimension as a vector space over $k((t))$.
 In this case also we apply Theorem~\ref{coefficient field case} to prove the claim.
\end{proof}

An immediate consequence of the theorem is that there exist non-Noetherian stable domains of Krull dimension $1$:

\index{stable domain!existence of}
\begin{cor} \label{stable corollary}
Let $k$ be   a field either of characteristic $0$ or of characteristic $p \ne 0$ such that $[k:k^p]$ is uncountable, and let $X$
be an indeterminate for $k$.  Then there exists a non-Noetherian
bad stable domain $R$ whose normalization is $k[[X]]$. 
\end{cor}

%\begin{proof}
%\end{proof}

%\begin{prop} \label{characteristic p case} Let $S$ be a DVR of characteristic $0$ with quotient field $F$, and suppose that the residue field of $S$ is a field with $p$ elements ($p$ a prime), and $pS$ is the maximal ideal of $S$.
%Then  there is a one-to-one correspondence between the subrings $R$ of $S$ having integral closure $S$ and the full $S$-submodules of $\Omega_{F/{\mathbb{Q}}}$.  Moreover, each such proper subring $R$ of $S$ is a bad stable domain, and is Noetherian of embedding dimension $\epsilon_R+1$ if and only if $\epsilon_R$ is finite.
%\end{prop}

\section{Local rings with normalization a complete DVR}

In this section we examine more closely the case in which an analytically ramified local Noetherian domain has normalization a complete DVR.  In particular, we use Matlis' theory of $Q$-rings  
%As discussed before Corollary~\ref{main theorem 1}, every one-dimensional analytically ramified local Noetherian domain has as an overring a bad $2$-generator.   
%
 to characterize in Theorem~\ref{Q1 ring char} the bad $2$-generator domains whose normalization is a complete DVR, and with this we recover in Corollary~\ref{Matlis corollary} a characterization due to Matlis of the analytically ramified local Noetherian domains having normalization a complete DVR.

In the article \cite{MQring}, Matlis  introduced the notion of a $Q$-ring to study the domains $R$ with quotient field $Q$ which have an $R$-module $L$ such that $\Hom_R(Q,L) = 0$ and $\Ext_R^1(Q,L) \cong Q$.  The motivation for studying such rings was a question from singular cohomology (ultimately answered in the negative) of whether the ring of integers possessed such a module.  Matlis proved in Theorem 2.1 of \cite{MQring} that the existence of such an $R$-module was equivalent to $R$ being what he termed a ``$Q$-ring,'' which we define momentarily.
The $Q$-rings fall into three classes, depending on the number of proper nonzero $h$-divisible $R$-submodules of $Q/R$, where an $R$-module $L$ is {\it $h$-divisible} \index{hdivis@$h$-divisible module} if $L$ is a homomorphic image of a divisible torsion-free $R$-module.  If $R$ is a $Q$-ring, then $Q/R$ has either $0$, $1$ or $2$ proper nonzero $h$-divisible $R$-submodules \cite[Theorem 3.4]{MQring}.

 An integral domain $R$ with quotient field $Q$ is a {\it $Q$-ring} if
$\Ext_R^1(Q,R) \cong Q$. The domain $R$ is a $Q_0$-, $Q_1$-, or $Q_2$-{\it ring} depending on whether $Q/R$ contains $0$, $1$ or $2$ proper nonzero $h$-divisible $R$-submodules.
In the case where $F$ denotes the quotient field of $R$, we still use the term $Q$-ring (as opposed to, say, $F$-ring) to describe the situation in which $\Ext_R^1(F,R) \cong F$.
Recall that the {\it rank} \index{rank of a module} of a torsion-free module $L$ over a domain $R$ with quotient field $F$ is the dimension of the $F$-vector space $F \otimes_R L$. The following lemma uses the notion of  completion in the ideal topology, which was discussed in (2.5): If $R$ is a domain, then $\wt{R} = \varprojlim R/rR$, where $r$ ranges over the nonzero elements of $R$.  
The rank of the completion of $R$ determines whether $R$ is a
$Q$-ring:

\begin{lem} \label{Matlis Q-ring char} {\em (Matlis \cite[Theorem 2.2]{MQring})} A domain $R$ is a $Q$-ring
if and only if the rank of the torsion-free $R$-module $\wt{R}$ is
$2$, where the completion is with respect to the ideal topology on $R$.
\end{lem}

\index{Qr@$Q_1$-ring}
We consider now $Q_1$-domains whose normalization is a DVR.
% Our \index{Matlis, E.}
%motivation is a result of Matlis which shows that if $R$ is a local
%Noetherian domain of Krull dimension $1$ that is not complete, then
%the normalization of $R$ is a complete DVR if and only if there
%is a $Q_1$-ring between $R$ and its quotient field \cite[Theorem
%6.7]{MQring}.  We recover this characterization below in Corollary~\ref{Matlis corollary}.

\begin{lem} \label{Matlis Q-ring lemma} {\em (Matlis \cite[Theorem 4.3]{MQring})} Let $R$ be a $Q_1$-domain with quotient field $F$, and let $B/R$ be the unique
proper nonzero $h$-divisible $R$-submodule of $F/R$.  Then $B$ is a ring and  there is
a unique minimal prime ideal $P$ of $\wt{R}$, where $\wt{R}$ is the
completion of $R$ in the ideal topology, such that  $\wt{R}/P \cong B$; $P
\cong \Hom_R(F/R,B/R)$; and $P^2 =0$.
\end{lem}

The assertion that $P^2 =0$ is not in the statement of Matlis'
theorem but it can be found in the proof. 
The next lemma follows from results in \cite{OlbGFF},
and can be viewed as another characterization of  bad
$2$-generator domains.

\index{stable domain!bad}
  \begin{lem} \label{Q-ring lead in} Let $R$ be  a
  bad stable domain with normalization $V$ and quotient field $F$.
 Then every $R$-submodule of $V/R$ admits a
  unique
  $V$-module structure extending the $R$-module structure,  and
   the following statements are equivalent.

  \begin{itemize}

  \item[(1)]  $V/R$ is an indecomposable $R$-module.

%\item[(2)]  $V/R \cong F/V$ as $V$-modules.

  \item[(2)]  $V/R$ has no proper nonzero divisible $R$-submodules.

  \item[(3)]  $V/R$ is the unique proper nonzero divisible $R$-submodule of
  $F/R$.

\index{$2$-generator ring!bad}
  \item[(4)]  $R$ is a bad $2$-generator ring.
\end{itemize}

\end{lem}

\begin{proof}  Since  $V/R$ is a torsion divisible $R$-module and  $R
\subseteq V$ is a quadratic extension, we may apply
Lemma~\ref{start}(4) to obtain that $V/R$ is a $V$-module, and from this it follows that this $V$-module
structure induces a $V$-module structure on every $R$-submodule of
$V/R$; see  \cite[Lemma 3.1]{OlbGFF}.

%(1) $\Leftrightarrow$ (2)   By Corollary~\ref{2 gen char local}, %$S/R \cong F/S$ as $S$-modules, and hence $S/R$ is a divisible %$R$-module. By Proposition~\ref{Matlis}, as a
%divisible $R$-module,  $S/R$ is also a divisible $S$-module. By
%assumption $S$ is a DVR, and hence every divisible $S$-module is an
%injective $S$-module.  In particular, $S/R$ is an injective
%$S$-module.  Let $N$ denote the maximal ideal of $S$. Since $S$ is %a
%DVR, a torsion $S$-module is injective and indecomposable if and
%only if it is isomorphic to $F/S$ (for example, combine
%\cite[Theorem 4.5]{M1} and \cite[Theorem 4]{M3}). Therefore, since
%$S/R$ is an injective $S$-module, we have that $S/R$ is an
%indecomposable $S$-module if and only if $S/R \cong F/S$ as
%$S$-modules.

(1) $\Rightarrow$ (2)  Suppose that  $B/R$ is a proper nonzero
divisible $R$-submodule  of $V/R$.  As above,  $B/R$ is a $V$-submodule of $V/R$, and   since $R$ and $V$ share the same quotient field, it follows that $B/R$ is a divisible $V$-submodule of $V/R$.
  Then since $V$ is a DVR,
$B/R$ is an injective $V$-module, and hence is a summand of $V/R$.
This shows that if $V/R$ is indecomposable, then there are no nonzero proper
divisible $V$-submodules of $V/R$.

(2) $\Rightarrow$ (3)  Suppose that $B$ is an $R$-submodule of $F$
containing $R$ such that $B/R$ is a proper nonzero divisible
$R$-submodule of $F/R$.  A straightforward argument shows that the divisibility of this module implies that  $B$ is a  ring.  But
 since the normalization of $R$ is $V$ and $R \subseteq B$, it follows that
 $V$ is a subring
of the normalization of $B$.  Yet $V$ is a DVR, so this forces $V$
to be the normalization of $B$, and hence $B \subseteq V$.  Thus
by (2), $B/R = V/R$.

(3) $\Rightarrow$ (1)  Suppose that $V/R=  (B_1/R) \oplus (B_2/R)$ for
some $R$-submodules $B_1$ and $B_2$ of $V$ containing $R$.   Then as an
image of $V/R$, each of  $B_1/R$ and $B_2/R$ is a divisible
$R$-module. Thus by (3), $B_1 = V$ or
$B_2 = V$, and hence $V/R$ is an indecomposable $R$-module.

(1) $\Rightarrow$ (4)  Assume (1).  
Then by (2.4),
$V/R$ is isomorphic to $F/V$ as $V$-modules.  It
follows that the $V$-submodules of $V/R$ form a chain under
inclusion.  But, as noted above,  every $R$-submodule of $V/R$ is
also a $V$-submodule, so the set of $R$-submodules of $V/R$ forms a
chain.  Let $M$ denote the maximal ideal of $R$, and let $R_1 = \{q
\in F:qM \subseteq M\}$.  Then since $M$ is a stable ideal, $M =
mR_1$ for some $m \in M$. If $R_1 =
R$, then $M$ is a principal ideal of $R$ and $R = V$, a contradiction to
the assumption that $R$ is a bad stable domain. Thus $R_1/R$ is a
nonzero $R/M$-vector space, and since $R_1/R \subseteq V/R$ and the
$R$-submodules of $V/R$ form a chain, it must be that $R_1/R$ has
dimension $1$ as a vector space. Therefore, there exists $x \in R_1$
such that $R_1 = xR + R$, and consequently, $M = mR_1 = mxR + mR$.
Thus $R$ is a stable domain with a maximal ideal that can be
generated by $2$ elements. Therefore, since the maximal ideal has reduction number $\leq 1$, the multiplicity of $R$ is at most $2$.  Consequently, 
 every ideal of $R$
can be generated by $2$ elements \cite[Theorem 1.1, p.~49]{Sally}.

(4) $\Rightarrow$ (1) Assume (4).  Since every ideal of $R$ can be generated by
2 elements, the $R$-submodules of $V/R$ form chain under inclusion (see for example \cite[Lemma 3.5]{OlRend}), and hence $V/R$ is indecomposable.  
\end{proof}

With the lemma, we characterize bad $2$-generator domains with
normalization  a complete DVR.

\index{$2$-generator ring!bad} \index{DVR!complete} \index{Qr@$Q_1$-ring}
\begin{thm} \label{Q1 ring char}
A ring   $R$ is a bad $2$-generator domain whose normalization $V$ is a complete
DVR if and only if $R$ is a one-dimensional quasilocal $Q_1$-ring  tightly dominated by $V$.
\end{thm}

\begin{proof}
Suppose that  $R$ is a bad $2$-generator domain with $V$ a complete DVR.  Then by Lemma~\ref{Q-ring lead in}, $V/R$ is
the unique proper nonzero divisible $R$-submodule of $F/R$ (where
$F=$ quotient field of $R$) and it follows that $V$ tightly dominates $R$. Also, by (2.5),
there exists a prime ideal $P$ of $\wt{R}$ such that $P^2 = 0$ and
$\wt{R}/P \cong V$. (Here the completion is with respect to the ideal topology.)
A theorem of Matlis shows that in this case, $P
\cong \Hom_R(F/V,V/R)$ \cite[Theorem 2.6]{M1}. But since $R$ is a bad $2$-generator domain,
$V/R \cong F/V$ (2.4), so that $P \cong
\Hom_R(F/V,F/V)$.  As a ring, $\Hom_R(F/V,F/V)$ is isomorphic to $\widehat{V}$ \cite[Theorem 2.2]{M1}, 
so we see
then that $P \cong \wt{V} \cong V$, since $V$ is a complete
DVR.  Thus since $P$ is isomorphic to the $R$-submodule $V$ of $F$,
 $P$ has rank $1$ as a torsion-free $R$-module, as does
$\wt{R}/P \cong V$, so necessarily $\wt{R}$ has rank $2$ as a
torsion-free $R$-module. By Lemma~\ref{Matlis Q-ring char},
$R$ is a $Q$-ring, and hence since $V/R$ is the unique proper
nonzero divisible $R$-submodule of $F/R$, $R$ is a $Q_1$-ring. (Recall here that divisible $= h$-divisible for quasilocal
domains of Krull dimension $1$, or, more generally, for Matlis domains \cite[Theorem VII.2.8, p.~253]{FS}.)

Conversely, suppose that $R$ is a $Q_1$-ring with $V/R$ a divisible
$R$-module.  Then by Lemma~\ref{Matlis Q-ring lemma}, there exists a
prime ideal $P$ of $\wt{R}$ such that $P^2 = 0$ and $\wt{R}/P \cong
V$.  Therefore, by (2.5), $R$ is a bad stable
domain.  Moreover, since $R$ is a $Q_1$-ring and $V/R$ is
$h$-divisible, it follows that $V/R$ contains no
proper nonzero divisible $R$-submodules.  Therefore, by
Lemma~\ref{Q-ring lead in}, $R$ is bad $2$-generator ring.
\end{proof}

\index{Matlis, E.}
From the theorem, we recover Matlis' characterization  of Noetherian domains with normalization a complete DVR \cite[Theorem 6.7]{MQring}:

\index{$2$-generator ring!bad} \index{Qr@$Q_1$-ring} \index{DVR!complete}
\begin{cor} \label{Matlis corollary}  {\em (Matlis)} Let $R$  be a one-dimensional analytically ramified local Noetherian domain.  Then the normalization
$V$ of $R$  is a complete DVR if and only if there exists a  $Q_1$-ring $T$ with $R \subseteq T \subseteq V$.  Moreover, the $Q_1$-ring $T$ can be chosen to be a bad $2$-generator ring.
\end{cor}

\begin{proof}  Suppose $V$ is a complete DVR.  As discussed before Corollary~\ref{main theorem 1}, 
 there exists a bad $2$-generator ring $T$  that birationally dominates $R$.  Since necessarily $T \subseteq V$ and hence $T$ has a local normalization, it must be that $V$ is the normalization of $T$.
Therefore,  by Theorem~\ref{Q1 ring char}, $T$ is a $Q_1$-ring.

  Conversely, suppose that $R \subseteq T \subseteq V$ with $T$ a $Q_1$-ring.  Let $B/T$ be the unique proper nonzero $h$-divisible $T$-submodule of $F/T$.  Then by Lemma~\ref{Matlis Q-ring lemma}, $\widehat{R}/P \cong B$ for some prime ideal $P$ of $\widehat{R}$.  Therefore, $B$ is a complete local Noetherian domain, and as such its normalization is a complete DVR (see for example \cite[Theorem 10.4, p.~93]{M1}).   Since $B \subseteq V$, this forces the normalization of $B$ to be $V$, and hence $V$ is a  complete DVR.
%
 % Conversely, with $T$ a bad $2$-generator ring, then $S/T$ is a divisible $T$-module (Proposition~\ref{ar stable char}), and by Lemma~\ref{Q-ring lead in}, $S/T$ is the unique proper nonzero divisible $R$-submodule of $F/T$, where $F$ is the quotient field of $T$.  Thus by Lemma~\ref{Matlis Q-ring lemma}, $\widehat{R}/P \cong S$ for some prime ideal $P$ of $\widehat{R}$, and hence $S$ is complete in the $R$-topology.  Therefore, since by Proposition~\ref{ar stable char}, $S$ is a DVR, we conclude that
  % $S$ is a complete DVR.
   \end{proof}

\index{Qr@$Q_1$-ring!existence of}
In \cite{MQring}, Matlis constructed a $Q_1$-ring in characteristic $2$, and asked whether there exist $Q_1$-rings in any other characteristics?  Using the results in this section, we find many other such examples in all characteristics:

\begin{cor}
Let $V$ be a complete DVR with residue field $k$, and suppose  that either:
\begin{itemize} \item[(a)] $k$ has characteristic $0$,

\item[(b)]  $V$ and $k$ have characteristic $p \ne 0$ and $[k:k^p]$ is uncountable, or

\item[(c)] $V = \widehat{{\mathbb{Z}}}_p$.
\end{itemize}
Then there exists a $Q_1$-ring $R$ such that $R$ is a bad $2$-generator ring and the normalization of $R$ is $V$.
\end{cor}

\begin{proof}
 In all cases,  Theorem~\ref{catalog complete DVR} implies that there exists a bad Noetherian stable domain whose normalization is $V$.  An application of Corollary~\ref{Matlis corollary} now completes the proof.
 \end{proof}

\medskip

{\bf Acknowledgment.} {I thank the referee for  helpful comments that improved the paper.}

\end{document}